\documentclass[]{amsart}

\usepackage{hyperref}

\usepackage{color} 

\usepackage[initials,nobysame,msc-links]{amsrefs}
\DefineSimpleKey{bib}{how}
\renewcommand{\PrintDOI}[1]{%
    \href{http://dx.doi.org/#1}{{\tt DOI:#1}}%
}
\renewcommand{\eprint}[1]{#1}
\BibSpec{book}{%
    +{}  {\PrintPrimary}                {transition}
    +{.} { \PrintDate}                  {date}
    +{.} { \textit}                     {title}
    +{.} { }                            {part}
    +{:} { \textit}                     {subtitle}
    +{,} { \PrintEdition}               {edition}
    +{}  { \PrintEditorsB}              {editor}
    +{,} { \PrintTranslatorsC}          {translator}
    +{,} { \PrintContributions}         {contribution}
    +{,} { }                            {series}
    +{,} { \voltext}                    {volume}
    +{,} { }                            {publisher}
    +{,} { }                            {organization}
    +{,} { }                            {address}
    +{,} { }                            {status}
    +{,} { \PrintDOI}                   {doi}
    +{,} { \PrintISBNs}                 {isbn}
    +{}  { \parenthesize}               {language}
    +{}  { \PrintTranslation}           {translation}
    +{;} { \PrintReprint}               {reprint}
    +{.} { }                            {note}
    +{.} {}                             {transition}
    +{}  {\SentenceSpace \PrintReviews} {review}
}
\BibSpec{collection.article}{%
    +{}  {\PrintAuthors}                {author}
    +{,} { \textit}                     {title}
    +{.} { }                            {part}
    +{:} { \textit}                     {subtitle}
    +{,} { \PrintContributions}         {contribution}
    +{,} { \PrintConference}            {conference}
    +{}  {\PrintBook}                   {book}
    +{,} { }                            {booktitle}
    +{,} { \PrintDateB}                 {date}
    +{,} { pp.~}                        {pages}
    +{,} { }                            {publisher}
    +{,} { }                            {organization}
    +{,} { }                            {address}
    +{,} { }                            {status}
    +{,} { \PrintDOI}                   {doi}
    +{,} { \PrintISBNs}                 {isbn}
    +{,} { available at \eprint}        {eprint}
    +{}  { \parenthesize}               {language}
    +{}  { \PrintTranslation}           {translation}
    +{;} { \PrintReprint}               {reprint}
    +{.} { }                            {note}
    +{.} {}                             {transition}
    +{}  {\SentenceSpace \PrintReviews} {review}
}
\BibSpec{misc}{%
  +{}{\PrintAuthors}  {author}
  +{,}{ \textit}      {title}
  +{.}{ }             {how}
  +{}{ available at \eprint}{eprint}
  +{}{ available at \url}{url}
  +{, }{ \PrintDateB} {date}
  +{,}{ }             {note}
  +{.}{}              {transition}
}

\usepackage{amssymb}
\usepackage{MnSymbol,slashed}
\usepackage{amscd}  
\usepackage[all]{xy} 

\usepackage{amsthm}
\theoremstyle{plain}
\newtheorem{thm}[]{Theorem}
\newtheorem{prop}[]{Proposition}
\newtheorem{lem}[prop]{Lemma}
\newtheorem{cor}[prop]{Corollary}

\theoremstyle{definition}

\theoremstyle{remark}
\newtheorem{rmk}[prop]{Remark}
\newtheorem{example}[prop]{Example}

\mathchardef\mhyphen="2D

\newcommand{\ensemble}[1]{\left\{ #1 \right\}}
\newcommand{\suchthat}{\mid}
\newcommand{\norm}[1]{\left\| #1 \right\|}
\newcommand{\absolute}[1]{\left| #1 \right|}

\newcommand{\cptops}{\mathcal{K}}

\newcommand{\N}{\mathbb{N}}
\newcommand{\Z}{\mathbb{Z}}
\newcommand{\Q}{\mathbb{Q}}
\newcommand{\C}{\mathbb{C}}
\newcommand{\R}{\mathbb{R}}
\newcommand{\T}{\mathbb{T}}

\newcommand{\bimod}[1]{\mathcal{#1}}

\newcommand{\KK}{\mathit{KK}}
\newcommand{\RKK}{\mathcal{R}\KK}

\newcommand{\liealg}[1]{\mathfrak{#1}}
\newcommand{\Adj}{\mathrm{Ad}}

\newcommand{\UnivEnv}{\mathcal{U}}
\newcommand{\Cliff}{\mathrm{Cl}}

\newcommand{\Disc}{\mathbb{D}}
\newcommand{\Tpltz}{\mathcal{T}}

\newcommand{\Lef}{\mathrm{Lef}}
\newcommand{\even}{\mathrm{ev}}
\newcommand{\odd}{\mathrm{odd}}
\newcommand{\qgrp}[1]{\mathbb{#1}}
\newcommand{\DD}{\mathsf{D}}
\newcommand{\DG}{\DD(\qgrp{G})}
\newcommand{\Hmod}[1]{\mathcal{#1}}
\newcommand{\multAlg}{\mathcal{M}}
\newcommand{\ideal}[1]{\mathcal{#1}}

\DeclareMathOperator{\Ad}{Ad}
\DeclareMathOperator{\End}{End}
\DeclareMathOperator{\Hom}{Hom}

\DeclareMathOperator{\Img}{Img}

\DeclareMathOperator{\rk}{rk}
\DeclareMathOperator{\Ev}{Ev}
\DeclareMathOperator{\ev}{ev}
\DeclareMathOperator{\Id}{Id}

\DeclareMathOperator{\Ind}{Ind}
\DeclareMathOperator{\Res}{Res}
\DeclareMathOperator{\Lie}{Lie}
\DeclareMathOperator{\grtensor}{\hat{\otimes}}

\DeclareMathOperator{\HmodCpt}{\mathcal{K}}
\newcommand{\SU}{\mathrm{SU}}
\newcommand{\U}{\mathrm{U}}

\begin{document}

\title{Equivariant comparison of quantum homogeneous spaces}
\author{Makoto Yamashita}
\date{May 17, 2012} 
\address{Dipartimento di Matematica,
Universit\`{a} degli Studi di Roma ``Tor Vergata''\\
Via della Ricerca Scientifica 1, 00133 Rome, Italy}
\curraddr{Department of Mathematics,
Ochanomizu University\\
Ohtsuka 2-1-1, Bunkyo-ku, 112-8610 Tokyo, Japan}
\thanks{Supported in part by the ERC Advanced Grant 227458 OACFT ``Operator Algebras and Conformal Field Theory''}
\email{yamashit@mat.uniroma2.it}
\begin{abstract}
 We prove the deformation invariance of the quantum homogeneous spaces of the $q$-deformation of simply connected simple compact Lie groups over the Poisson--Lie quantum subgroups, in the equivariant $\KK$-theory with respect to the translation action by maximal tori.  This extends a result of Neshveyev--Tuset to the equivariant setting.  As applications, we prove the ring isomorphism of the $K$-group of $G_q$ with respect to the coproduct of $C(G_q)$, and an analogue of the Borsuk--Ulam theorem for quantum spheres.
\end{abstract}

\maketitle

\section{Introduction}

After the seminal works of Woronowicz~\cite{MR901157} and Podle\'{s}~\cite{MR919322}, comodule algebras over quantum groups became known to be a rich class of `noncommutative spaces' in the study of operator algebras and noncommutative geometry.  One geometrically interesting generalization of their results is given by the $q$-deformation $G_q$ of a simply connected simple compact Lie group $G$, and the quantum homogeneous spaces of such quantum groups developed by Reshetikhin--Takhtadzhyan--Faddeev~\cite{MR1015339} and Vaksman--So{\u\i}bel{\cprime}man~\cite{MR1086447}.

The C$^*$-algebras obtained this way can be thought as continuous deformation of the commutative algebra of continuous functions on a homogeneous space of $G$, which is an ordinary compact Riemannian manifold.  Standing on this viewpoint, Nagy~\cite{MR1601237} took a $\KK$/$E$-categorical approach to compute the $K$-groups of $G_q$ by utilizing continuous fields of C$^*$-algebras over the parameter space with fiber $C(G_q)$.  His method was recently generalized to quantum homogeneous spaces over Poisson--Lie quantum subgroups $K^{S,L}_q$ by Neshveyev--Tuset~\cite{arXiv:1103.4346}.  The main idea is to reduce the comparison of $K$-groups to the case of deformation quantization of open discs which appear as the symplectic leaves in the homogeneous space~\cite{MR1614943}.

The categorical structure of comodule algebras over such quantum groups is also interesting in connection with the Baum--Connes problem of quantum groups due to Meyer--Nest~\citelist{\cite{MR2193334}\cite{MR2339371}}.  The analogue of the strong Baum--Connes conjecture for the dual of Hodgkin groups allows us to cast a new light on early results of Hodgkin~\cite{MR0478156}, McLeod~\cite{MR557175}, and Snaith~\cite{MR0309115} concerning the equivariant $K$-groups of $G$-spaces.  Then, it is natural to expect that the quantum homogeneous spaces are crucial in understanding of more general comodule algebras over $G_q$.  Indeed, the equivariant $\KK$-theory of the standard Podle\'{s} sphere plays a central role in the proof of the strong Baum--Connes conjecture for $\widehat{\SU}_q(2)$ by Voigt~\cite{arXiv:0911.2999}.

The main result of this paper is that the algebra $C(G_q/K^{S,L}_q)$ of the quantum homogeneous space is equivariantly $\KK$-equivalent to the classical one $C(G/K^{S,L})$ with respect to the translation action of the maximal torus (Theorem~\ref{thm:q-grp-bundle-eval-T-T-equivar-equiv}).  This extends a result of~\cite{arXiv:1103.4346} to the torus equivariant setting.  The proof occupies Section~\ref{sec:equivar-compar-quant-g-sp}, where we adapt the argument of~\cite{arXiv:1103.4346} to the equivariant case using the equivariant Universal Coefficient Theorem~\cite{MR849938} and the triangulated structure of the equivariant $\KK$-category~\cite{MR2193334}.

We observe two applications of our main theorem in Section~\ref{sec:applications}.  The first one concerns the ring structure of $K^*(C(G_q))$ with respect to the coproduct on $C(G_q)$ (Theorem~\ref{thm:ring-structure-G-q}).  By the continuous field comparison method, we prove that $K^*(C(G_q))$ is isomorphic to $K_*(G)$ as a ring.

The second application is an analogue of the Borsuk--Ulam theorem for the quantum spheres (Theorem~\ref{thm:quant-borsuk-ulam}), conjectured by Baum and Hajac~\cite{baum-hajac-galois-klein}.  The theorem states that there is no equivariant homomorphism from $C(S^n_{q})$ to $C(S^{n+1}_q)$ with respect to the antipodal action of $C_2 = \Z/2\Z$.  The proof, which is modelled on the argument of Casson--Gottlieb~\cite{MR0436144}, makes use of the Lefschetz number of equivariant $\KK$-morphisms.  The crucial point is that the antipodal action on a odd-dimensional sphere is given by a restriction of a $\U(1)$-action, which comes from the homogeneous space description of $S^n_q$ as $\SU_q(n)/\SU_q(n-1)$.

\paragraph{Acknowledgment} A major part of this research was carried out under the support of the Marie Curie Research Training Network MRTN-CT-2006-031962 in Noncommutative Geometry, EU-NCG.  The author thanks Sergey Neshveyev and Wojciech Szymanski for fruitful exchanges which were crucial to the development of the research.  This paper was written during the author's stay at the Mathematical Sciences Institute, Australian National University.  He would like to thank them, particularly Alan Carey and Adam Rennie, for their hospitality.  He is also grateful to Piotr M. Hajac, from whom he learned the Borsuk--Ulam problems during his lecture at the ANU in August 2011.  Finally, he would like to thank the referees for numerous suggestions which greatly improved the presentation of the paper, and also for pointing out a critical error in the early version of Section~4.1.

\section{Preliminaries}

Throughout the paper $G$ denotes a simply connected simple compact Lie group and $\liealg{g}$ denotes its Lie algebra.  We fix a maximal torus $T$ of $G$, and let $W$ denote the associated Weyl group.  We also fix a positive root system $P$ of $(G, T)$ and let $\Pi = \ensemble{\alpha_i \suchthat i = 1, \ldots, \rk{G}}$ denote the corresponding set of the simple roots.  The associated length function on $W$ is denoted by $l(w)$, and the longest element by $w_0$, whose length is denoted by $m_0$.  By abuse of the notation we sometimes use $\Pi$ as the index set $\ensemble{1, \ldots, \rk G}$ for the $\alpha_i$.  The Cartan matrix $2 (\alpha_i, \alpha_j)/(\alpha_i, \alpha_i)$ of $\liealg{g}$ is denoted by $(a_{i, j})_{i, j \in \Pi}$.  

When $A$ is a C$^*$-algebra, we let $\multAlg(A)$ denote its multiplier algebra.

If $A$ and $B$ are C$^*$-algebras, $A \otimes B$ denotes their minimal tensor product unless otherwise stated.  When $\Hmod{E}$ (resp. $\Hmod{F}$) is a right Hilbert $A$-module (resp. right Hilbert $B$-module), $\Hmod{E} \otimes \Hmod{F}$ denotes their tensor product as a right Hilbert $A \otimes B$-module.

The closed unit interval $[0, 1]$ is denoted by $I$.  We let $\N$ be the set of nonnegative integers and $\N_+ = \N \setminus \ensemble{0}$.  For each subset $X$ of $\N$, we let $\chi_X$ denote the orthogonal projection onto the subspace $\ell^2 X$ of $\ell^2 \N$.

When $A$ is a continuous field of C$^*$-algebras over $I$ (or the half open interval $(0, 1]$) whose fiber at $q \in I$ is $A_q$, and $X$ is a subset of $I$, we let $\Gamma_X(A_q)$ denote the C$^*$-algebra of the continuous sections over $X$.

\subsection{Noncommutative disks}

For $0 \le q < 1$, let $C(\overline{\Disc}_q)$ denote the C$^*$-algebra of noncommutative closed disc, which is the universal algebra generated by $Z_q$ satisfying
\begin{equation*}
1 - Z_q^* Z_q = q^2 (1 - Z_q Z_q^*).
\end{equation*}
Then $C(\overline{\Disc}_q)$ can be faithfully represented on $\ell^2 \N$ by setting
\begin{equation}
\label{eq:q-disc-Toeplitz-corr}
Z_q e_n = \sqrt{1 - q^{2(n+1)}} e_{n + 1} \quad (n \in \N).
\end{equation}
The algebra of operators generated by this representation is equal to the Toeplitz algebra $\Tpltz$ on $\ell^2 \N$ generated by the unilateral shift operator $S\colon e_n \mapsto e_{n + 1}$.  There is a homomorphism $C(\overline{\Disc}_q) \rightarrow C(S^1)$ given by $Z_q \mapsto z$.  The kernel $C_0(\Disc_q)$ of this homomorphism is isomorphic to the algebra $\mathcal{K}$ of the compact operators.  Under the identification $C(\overline{\Disc}_q) \simeq \Tpltz$, this corresponds to the standard extension $\cptops \rightarrow \Tpltz \rightarrow C(S^1)$ induced by the map $S \mapsto z$.

For $q = 1$, we formally put $C(\overline{\Disc}_1) = C(\overline{\Disc})$ and $C_0(\Disc_1) = C_0(\Disc)$, where $\overline{\Disc} = \ensemble{z \in \C \suchthat \absolute{z} \le 1}$ and $\Disc = \ensemble{z \in \C \suchthat \absolute{z} < 1}$.

\subsection{Algebras representing \texorpdfstring{$q$}{q}-deformations}

We briefly review the definition of the compact quantum group $G_q$ for $0 < q < 1$ and related constructs.  Overall we adopt the convention of Neshveyev--Tuset~\citelist{\cite{arXiv:0711.4302}\cite{arXiv:1103.4346}}.

Recall that the quantized universal enveloping algebra $\UnivEnv_q(\liealg{g})$ of $\liealg{g}$ is generated by the generators $E_i, F_i, K_i^{\pm 1}$ for $i \in \Pi$ subject to the relations
\begin{align}
\label{eq:q-env-def-torus-rel}
[K_i, K_j] &= 0, &
K_i E_j K_i^{-1} &= q_i^{a_{i, j}} E_j, &
K_i E_j K_i^{-1} &= q_i^{a_{i, j}} E_j
\end{align}
where $q_i = q^{(\alpha_i, \alpha_i)/2}$, and
\[
[E_i, F_j] = \delta_{i, j} \frac{K_i - K_i^{-1}}{q_i - q_i^{-1}},
\]
together with the $q$-analogues of the Serre relations.  It is a Hopf algebra by the coproduct $\hat{\Delta}_q$, whose values on the above generators are given by
\begin{align*}
\hat{\Delta}_q(K_i) &= K_i \otimes K_i,&
\hat{\Delta}_q(E_i) &= E_i \otimes 1 + K_i \otimes E_i,&
\hat{\Delta}_q(F_i) &= F_i \otimes K_i^{-1} + 1 \otimes F_i.
\end{align*}
There is a compatible $*$-algebra structure on $\UnivEnv_q(\liealg{g})$ defined by
\begin{align*}
K_i^* &= K_i, &
E_i^* &= F_i K_i,&
F_i^* &= K_i^{-1} E_i.
\end{align*}

Let $V$ be a representation of $\UnivEnv_q(\liealg{g})$.  For each weight $\lambda$, let $V(\lambda)$ denote the subspace of $V$ spanned by the vectors $v$ satisfying $K_i v = q^{\lambda(h_i)} v$ for any $i$.  Then, $V$ is said to be \textit{integrable} (\textit{admissible} in~\cite{arXiv:0711.4302}) if it is equal to the direct sum of the $V(\lambda)$ for $\lambda \in P$, and is spanned by its finite-dimensional subrepresentations.

Let $P_+$ be the set of dominant integral weights of $\liealg{g}$.  The finite-dimensional integrable irreducible representations of $\UnivEnv_q(\liealg{g})$ are classified by their highest weights, which are elements of $P_+$.  We let $V^{(q)}_\lambda$ be the irreducible representation with the highest weight $\lambda$.  Furthermore we let $\UnivEnv(G_q)$ be the algebra $\prod_{\lambda \in P_+} M(V_\lambda)$.  Then $\UnivEnv_q(\liealg{g})$ can be regarded as an appropriate dense subalgebra of $\UnivEnv(G_q)$.

The space $\C[G_q]$ of the matrix coefficients of finite-dimensional integrable representations of $\UnivEnv_q(\liealg{g})$ is a $*$-Hopf algebra with a unique faithful Haar state.  We let $(C(G_q), \Delta_q)$ denote its C$^*$-algebraic closure.  Its dual quantum group $(C^*(G_q), \hat{\Delta}_q)$ is spanned by the algebraic direct sum $\oplus_{\lambda \in P_+} M(V_\lambda)$ of matrix algebras.

Unless otherwise stated, we consider the $\sigma(\UnivEnv(G_q), \C[G_q])$-topology on $\UnivEnv(G_q)$ in the following.  A family $(\phi^q)_{q \in (0, 1]}$ of $*$-isomorphisms said to be \textit{continuous} if it is compatible with the canonical identification $Z(\UnivEnv(G_q)) \simeq \prod_{\lambda \in P_+} \C \simeq Z(\UnivEnv(G))$, and the families $(\phi_q(E_i))_{q \in (0, 1]}$, $(\phi_q(F_i))_{q \in (0, 1]}$, and $(\phi_q(\log_{q_i}(K_i)))_{q \in (0, 1]}$ are continuous in $M(V_\lambda)$ for arbitrary $i$ (that is, the representation under $\pi_\lambda$ of these families are continuous in $M(V_\lambda)$ for any $\lambda$).  We always assume that $\phi^1$ is equal to the identity map of $\UnivEnv(G)$.

\subsection{Poisson--Lie subgroups}
\label{sec:poisson-subgrp}

We review the Poisson--Lie subgroup $K^{S, L} \subset G$ and its quantization as a quantum subgroup of $G_q$ defined by the auxiliary data $(S, L)$, following the treatment of~\cite{arXiv:1103.4346}.  First, $S$ is a subset of $\Pi$.  We let $\liealg{g}_S$ be the subalgebra of $\liealg{g}$ generated by $E_i$ and $F_i$ for $\alpha_i \in S$.  We let $\tilde{K}^S$ denote the closed connected subgroup of $G$ characterized by $\Lie(\tilde{K}^S) = \liealg{g}_S$.  Next, let $(w_i)_{i = 1}^{\rk G}$ be the fundamental weights associated to $P$, and $P^c(S)$ be the subgroup of the weight lattice spanned by the fundamental weights $w_i$ for $\alpha_i \nin S$.  Then $L$ is a subgroup of $P^c(S)$.  Since each weight can be thought as a character on $T$, we may take  the joint kernel $T_L \subset T$ of $L$.  Then, $K^{S, L}$ is defined to be the closed subgroup of $G$ generated by $\tilde{K}^S$ and $T_L$.

Let $\UnivEnv(K^{S, L}_q)$ be the closed subalgebra of $\UnivEnv(G_q)$ generated by $T_L$ and $E_i, F_i$ for $\alpha_i \in S$.  The image of $\C[G_q]$ under the transpose map $\UnivEnv(G_q)^* \rightarrow \UnivEnv(K^{S, L}_q)^*$ is a $*$-algebra denoted by $\C[K^{S, L}_q]$.

The algebra $\C[G_q/K^{S,L}_q]$ is by definition the subalgebra of $\C[G_q]$ consisting of the elements $f$ satisfying $\iota \otimes \pi \circ \Delta_q(f) = f \otimes 1$ for the restriction map $\pi\colon \C[G_q] \rightarrow \C[K^{S,L}_q]$.  The closure of $\C[G_q/K^{S,L}_q]$ in $C(G_q)$ is denoted by $C(G_q/K^{S,L}_q)$.  This notation is consistent with the usual notion of functions over the homogeneous space $G/K^{S, L}$ at $q = 1$.

For each $t \in T$, we have the evaluation map $\ev_t \colon C(G_q) \rightarrow \C$.  Let $L_t$ and $R_t$ denote the translation actions on $C(G_q)$ defined by
\begin{align*}
  L_t(f) &= \left ( \ev_t \otimes \iota \right ) \Delta_q(f),&
  R_t(f) &= \left ( \iota \otimes \ev_t \right ) \Delta_q(f).
\end{align*}
We call $\Adj_t = L_t R_{-t}$ the adjoint action of $T$.  The left translation action of $T$ on $C(G_q)$ restricts to $C(G_q/K^{S,L}_q)$.  Since the subalgebra $\UnivEnv(K^{S, L}_q)$ is normalized by $T$, we also have a right translation action of $T$ on $C(G_q/K^{S,L}_q)$, which is trivial on $T_L$.  Hence the right translation action can be thought as an action of $T/T_L$.

\begin{example}
 If $S = \emptyset$ and $L = 0$, the quantum subgroup $K^{S, L}_q$ is equal to $T$.  Hence $C(G_q/K^{S,L}_q)$ is the algebra $C(G_q/T)$ of quantum flag manifold.
\end{example}

\begin{example}
\label{example:quantum-sphere-as-homogen-su-n-su-n-1}
 Let $G = \SU(n)$ and $\Pi = \ensemble{\alpha_1, \ldots, \alpha_{n-1}}$ with the standard linear ordering satisfying $(\alpha_i, \alpha_i) = 2, (\alpha_i, \alpha_{i\pm1}) = -1$.  Then, for $S = \ensemble{\alpha_1, \ldots, \alpha_{n-2}}$ and $L = P^c(S)$, we have $K^{S,L} = \SU(n-1)$ which is embedded into the upper-left corner of $G$.  Thus $K^{S, L}_q$ can be identified with the quantum subgroup $\SU_q(n-1)$ of $G_q = \SU_q(n)$.  Then the associated homogeneous space $G_q/K^{S,L}_q$ is the odd-dimensional quantum sphere $S^{2n-1}_q$ introduced by Vaksman--So{\u\i}bel{\cprime}man~\cite{MR1086447}.  The right translation by $T/T_L = \U(1)$ is the gauge action for the standard generators $z_1, \ldots, z_n$ of $C(S^{2n-1}_q)$.
\end{example}

\subsection{Continuous fields of quantum spaces}
\label{sec:cont-field-quant-homog-sps}

We review the continuous field of quantum spaces, following the treatment of~\cite{arXiv:1103.4346}.

Let $\Gamma_{I}(C(\overline{\Disc}_q))$ denote the universal algebra $C^*(Q, Z)$ generated by a positive contraction $Q$ and another element $Z$ satisfying
\begin{align*}
\norm{Z} &= 1, &
Q Z &= Z Q, &
1 - Z^* Z &= Q^2 (1 - Z Z^*).
\end{align*}
The inclusion of $C(I) \simeq C^*(Q)$ into $C^*(Q, Z)$ defines the structure of a $C(I)$-algebra.  This way, the algebras $(C(\overline{\Disc}_q))_{q \in I}$ form a continuous field of C$^*$-algebras such that $q \mapsto Z_q$ is a continuous section.  Next, the algebra $\Gamma_I(C_0(\Disc_q))$ is defined as the kernel of the homomorphism of $\Gamma_I(C_0(\overline{\Disc}_q))$ onto $C(I) \otimes C(S^1)$ given by $Z \rightarrow z \in C(S^1)$.  This defines a continuous field with fiber $(C_0(\Disc_q))_{q \in I}$.  By nuclearity of fibers, the tensor products of these $C(I)$-algebras over $C(I)$ become again continuous fields in an unique way.

Next, let us recall the field of quantum homogeneous spaces.  Let $\phi^q\colon \UnivEnv(G_q) \rightarrow \UnivEnv(G)$ be a continuous family of $*$-algebra isomorphisms for $q \in (0, 1]$ which respect the highest weight of the irreducible modules.  By transposition, we obtain a family of coalgebra isomorphisms $\phi^{q \#}$ from $\C[G]$ to $\C[G_q]$.  If we fix an closed subinterval $J$ of $(0, 1]$, the algebras $(C(G))_{q \in J}$ admit a unique structure of continuous field such that $(f(q) (\phi^q)^\#(a))_{q \in J}$ is a continuous section for each $f  \in C(J)$ and $a \in \C[G]$~\cite{arXiv:1102.0248}*{Theorem~1.2}.

In other words, the algebraic tensor product $C(J) \otimes \C[G]$ becomes a $*$-algebra by
\[
(f^1 \otimes a)  ( f^2 \otimes b)(q) = f^1(q) f^2(q) (\phi^{q \#})^{-1}(\phi^{q \#}(a) \phi^{q \#}(b)).
\]
There is a unique pre-C$^*$-norm on this algebra, such that the quotient at $q \in J$ of its C$^*$-algebraic closure is equal to $C(G_q)$.  We let $\Gamma_J(C(G_q))$ denote the closure of $C(J) \otimes \C[G]$.  This is a $C(J)$-algebra with fiber $C(G_q)$ at $q \in J$.

Let $K$, $S$, and $L$ be as in the previous section.  We may further assume that $\phi^q(\UnivEnv(K^{S, L}_q)) = \UnivEnv(K^{S, L})$.  This implies $\phi^{q \#}(\C[G/K^{S, L}]) = \C[G_q/K^{S, L}_q]$, and that the subalgebra of $\Gamma_J(C(G_q))$ spanned by $C(J) \otimes \C[G/K^{S, L}]$ is a continuous subfield whose fiber at $q$ is $C(G_q/K^{S, L}_q)$~\cite{arXiv:1103.4346}*{Proposition~5.3}.

\subsection{Braided tensor products}
\label{sec:braded-tensor-prod}

Next, we review the basic facts on the Yetter--Drinfeld algebras and their braided tensor products, following the convention of Nest--Voigt~\cite{MR2566309}.

Let $\qgrp{G}$ be a compact quantum group.  We let $C^r(\qgrp{G})$ and $C^*(\qgrp{G})$ denote the reduced function algebra and the convolution C$^*$-algebra associated with $\qgrp{G}$.

Let $W$ denote the multiplicative unitary of $\qgrp{G}$.  It is a unitary operator on $L^2(\qgrp{G}) \otimes L^2(\qgrp{G})$ satisfying the pentagonal equation $W_{12}W_{13}W_{23} = W_{23} W_{12}$ and
\begin{align*}
C^r(\qgrp{G}) &= \overline{\ensemble{\iota \otimes \omega(W) \suchthat \omega \in B(L^2(\qgrp{G}))_*}}^{\norm{\cdot}},&
C^*(\qgrp{G}) &= \overline{\ensemble{\omega \otimes \iota(W) \suchthat \omega \in B(L^2(\qgrp{G}))_*}}^{\norm{\cdot}}.
\end{align*}
The operator $W$ actually belongs to $\multAlg(C^r(\qgrp{G}) \otimes C^*(\qgrp{G}))$, and it implements the coproducts of $C^r(\qgrp{G})$ and $C^*(\qgrp{G})$ by
\begin{align*}
\Delta&\colon C^r(\qgrp{G}) \rightarrow C^r(\qgrp{G}) \otimes C^r(\qgrp{G}), & x &\mapsto W^* (1 \otimes x) W,\\
\hat{\Delta}&\colon C^*(\qgrp{G}) \rightarrow \multAlg(C^*(\qgrp{G}) \otimes C^*(\qgrp{G})), & y &\mapsto \Sigma W (y \otimes 1) W^* \Sigma,
\end{align*}
where $\Sigma$ is the flip map.

The Drinfeld double $\DG$ of $\qgrp{G}$ is a locally compact quantum group whose function algebra is given by $C_0^r(\DG) = C^r(\qgrp{G}) \otimes C^*(\qgrp{G})$ endowed with the coproduct
\[
\Delta_{\DG} = (\Sigma \circ \Adj_W)_{2 3} \circ (\Delta \otimes \hat{\Delta}).
\]
To have a structure of $C_0^r(\DG)$-algebra on a C$^*$-algebra $A$ is equivalent to have that of a \textit{$\qgrp{G}$-Yetter--Drinfeld algebra} on $A$, which is given by coactions
\begin{align}
\label{eq:YD-alg-coaction-components}
\lambda\colon& A \rightarrow \multAlg(C^*(\qgrp{G}) \otimes A),&
\alpha\colon& A \rightarrow C^r(\qgrp{G}) \otimes A,
\end{align}
such that the diagram
\[
\begin{CD}
A @>{\lambda}>> \multAlg(C^*(\qgrp{G}) \otimes A) @>{\iota \otimes \alpha}>> \multAlg(C^*(\qgrp{G}) \otimes C^r(\qgrp{G}) \otimes A) \\
@VV{\alpha}V @. @VV{\Sigma_{12}}V \\
C^r(\qgrp{G}) \otimes A @>{\iota \otimes \lambda}>> \multAlg(C^r(\qgrp{G}) \otimes C^*(\qgrp{G}) \otimes A) @>{\Adj_{W_{1 2}}}>> \multAlg(C^r(\qgrp{G}) \otimes C^*(\qgrp{G}) \otimes A)
\end{CD}
\]
commutes.  Let $a * \psi = \psi \otimes \iota (\lambda(a))$ denote the right action of $\psi \in \C[\qgrp{G}]$ on $A$ induced by $\lambda$, and let $S$ be the full antipode of $\C[\qgrp{G}]$.  Then the above compatibility condition between $\lambda$ and $\alpha$ can be written as
\begin{equation}
\label{eq:yetter-drinfeld-condi-by-action}
\alpha( a * \psi ) = (\psi_{(1)}\alpha(a)_1 S(\psi_{(3)})) \otimes (\alpha(a)_2 * \psi_{(2)}).
\end{equation}

Let $A$ be a $\qgrp{G}$-Yetter--Drinfeld algebra, whose coactions are respectively given by $\lambda_A$ and $\alpha_A$ as in~\eqref{eq:YD-alg-coaction-components}.  When $B$ is a C$^*$-algebra with a coaction $\alpha_B$ of $C^r(\qgrp{G})$, we obtain the \textit{braided tensor product} $A \boxtimes_{\qgrp{G}} B$~\citelist{\cite{MR2182592}\cite{MR2566309}*{Section~3}}, which is the C$^*$-algebra linearly spanned by the operators of the form $\lambda_A(a) \alpha_B(b)$ considered as endomorphisms of the right Hilbert $A \otimes B$-module $L^2(\qgrp{G}) \otimes A \otimes B$.  There is a structure of a $\qgrp{G}$-algebra on $A \boxtimes_{\qgrp{G}} B$ given by the coaction
\[
A \boxtimes_{\qgrp{G}} B \rightarrow C^r(\qgrp{G}) \otimes A \boxtimes_{\qgrp{G}} B, \quad \lambda_A(a) \alpha_B(b) \mapsto \lambda_A(\alpha_A(a))_{1 2 3} \alpha_B(\alpha_B(b))_{1 2 4}.
\]
Suppose in addition that there is a coaction $\lambda_B$ of $C^*(\qgrp{G})$ on $B$ satisfying the $\qgrp{G}$-Yetter--Drinfeld algebra condition.  Then $A \boxtimes_{\qgrp{G}} B$ becomes a $\qgrp{G}$-Yetter--Drinfeld algebra by the above coaction of $C^r(\qgrp{G})$ and the coaction of $C^*(\qgrp{G})$ defined by
\[
A \boxtimes_{\qgrp{G}} B \rightarrow C^*(\qgrp{G}) \otimes A \boxtimes_{\qgrp{G}} B, \quad \lambda_A(a) \alpha_B(b) \mapsto \lambda_A(\lambda_A(a))_{1 2 3} \alpha_B(\lambda_B(b))_{1 2 4}.
\]

\begin{rmk}
\label{rmk:br-tensor-func-alg-with-itself}
The algebra $C^r(\qgrp{G})$ is a $\qgrp{G}$-Yetter--Drinfeld algebra.  The coaction $\alpha$ in~\eqref{eq:YD-alg-coaction-components} is given by the coproduct $\Delta$ of $C^r(\qgrp{G})$, and the one $\lambda$ of $C^*(\qgrp{G})$ is by the left adjoint coaction.  The braided tensor product of $C^r(\qgrp{G})$ with itself is isomorphic to $C^r(\qgrp{G}) \otimes C^r(\qgrp{G})$, via the isomorphism
\begin{equation}
\label{eq:braid-tensor-isom-usual-tensor}
C^r(\qgrp{G}) \boxtimes C^r(\qgrp{G}) \rightarrow C^r(\qgrp{G}) \otimes C^r(\qgrp{G}), \quad\lambda(a)_{1 2} \delta(b)_{1 3} \mapsto (a \otimes 1) \Delta(b).
\end{equation}
\end{rmk}

\subsection{Induction of coaction}
\label{sec:induc-coact}

Finally, let us review the theory of induction of coactions of compact quantum groups, following the treatment of Vaes~\cite{MR2182592} and~\cite{MR2566309}.  In this section $\qgrp{G}$ denotes a compact quantum group and $\qgrp{K}$ denotes its closed quantum subgroup.  We consider the reduced coactions of $\qgrp{G}$ and $\qgrp{K}$ throughout (although we formulate the results in this generality, later we will only need the case where $\qgrp{G} = G_q$ and $\qgrp{K} = T$).

In this section we let $M$ denote the von Neumann algebra $L^\infty(\qgrp{G})$ represented on $L^2(\qgrp{G})$ by the GNS representation of the Haar state, and $M'$ its commutant.  We let $R_{\qgrp{K}}$ denote the right translation coaction $M \rightarrow M \mathop{\overline{\otimes}} L^\infty(\qgrp{K})$. 

Let $A$ be a C$^*$-algebra endowed with a coaction $\alpha$ of $\qgrp{K}$, and put
\[
\tilde{A} = \ensemble{X \in \multAlg(\HmodCpt(L^2(\qgrp{G})) \otimes A) \suchthat X \in (M' \otimes \C)', R_{\qgrp{K}} \otimes \iota(X) = \iota \otimes \alpha(X)}.
\]
We consider the natural representation of $\tilde{A}$ by endomorphisms of the right C$^*$-$A$-module $L^2(\qgrp{G}) \otimes A$.  As before, $W$ denotes the fundamental unitary of $\qgrp{G}$.  Then we may take a homomorphism
\[
\Delta \otimes \iota \colon \tilde{A} \rightarrow \End (L^2(\qgrp{G}) \otimes L^2(\qgrp{G}) \otimes A)_{A}, \quad X \mapsto \Adj_{W^*_{1 2}}(X_{2 3}).
\]

The induced algebra $\Ind^{\qgrp{G}}_{\qgrp{K}} A$ can be characterized as the unique C$^*$-subalgebra of $\tilde{A}$ satisfying the following conditions~\cite{MR2182592}*{Theorem~7.2}.
\begin{itemize}
\item $\Delta \otimes \iota$ restricts to a homomorphism $\Ind^{\qgrp{G}}_{\qgrp{K}} A \rightarrow C^r(\qgrp{G}) \otimes \Ind^{\qgrp{G}}_{\qgrp{K}} A$, which is a continuous coaction of $C^r(\qgrp{G})$.
\item $\Delta \otimes \iota\colon \tilde{A} \rightarrow \multAlg(\HmodCpt(L^2(\qgrp{G})) \otimes \Ind^{\qgrp{G}}_{\qgrp{K}} A)$ is well defined and strictly continuous on the unit ball of $\tilde{A}$.
\item the representation of $\Ind^{\qgrp{G}}_{\qgrp{K}} A$ on $L^2(\qgrp{G}) \otimes A$ is nondegenerate, i.e. the subspace $(\Ind^{\qgrp{G}}_{\qgrp{K}} A) (L^2(\qgrp{G}) \otimes A)$ is dense in $L^2(\qgrp{G}) \otimes A$.
\end{itemize}
For the second condition, we are considering the restriction of the strict topology on $\multAlg(\HmodCpt(L^2(\qgrp{G})) \otimes A)$ to $\tilde{A}$.

The following description of induced algebra should have been known to experts, but we include it for the sake of completeness.

\begin{prop}
\label{prop:qgrp-coaction-ind-as-subalg-tens-prod}
Let $A$ be a $\qgrp{K}$-algebra.  Then the C$^*$-subalgebra
\begin{equation}
\label{eq:induction-for-cpt-subgrp}
D = \ensemble{X \in C^r(\qgrp{G}) \otimes A \suchthat R_{\qgrp{K}} \otimes \iota(X) = \iota \otimes \alpha(X)}
\end{equation}
of $\tilde{A}$ satisfies the above conditions for $\Ind^{\qgrp{G}}_{\qgrp{K}}(A)$.
\end{prop}

\begin{proof}
The first condition is satisfied by the commutativity of the left and right translations.  The third one follows from the fact that $D$ contains $C^r(\qgrp{G}/\qgrp{K}) \otimes A^{\alpha}$ whose action is already nondegenerate on $L^2(\qgrp{G}) \otimes A$, thanks to the compactness of $\qgrp{G}$.

Let us observe that the second condition also holds.  First, we regard $\cptops(L^2(\qgrp{G}))$ as the reduced crossed product $C^*(\qgrp{G}) \rtimes_{\hat{\Delta}} C^r(\qgrp{G})$ of $C^*(\qgrp{G})$ by itself with respect to the coaction $\hat{\Delta}$.  The dual coaction is given by $\Delta$ on the copy of $C^r(\qgrp{G})$ and the trivial one on the copy of $C^*(\qgrp{G})$.  Hence we obtain a C$^*$-algebra homomorphism
\[
\Delta \otimes \iota\colon \cptops(L^2(\qgrp{G})) \otimes A \rightarrow \cptops(L^2(\qgrp{G})) \otimes C^r(\qgrp{G}) \otimes A,
\]
which induces a strictly continuous homomorphism between the multiplier algebras.  The restriction of this homomorphism to $\tilde{A}$ agrees with $\Adj_{W_{1 2}}$.  Finally, it can be easily seen that the image of $\tilde{A}$ under $\Delta \otimes \iota$ is contained in $\multAlg(\cptops(L^2(\qgrp{G})) \otimes D)$.
\end{proof}

We note that, if $\qgrp{K}$ is an ordinary compact group, the above construction is equal to taking the fixed point algebra $(C^r(\qgrp{G}) \otimes A)^{\qgrp{K}}$ under the diagonal action $k.(f \otimes a) = R_{k^{-1}}(f) \otimes \alpha_k(a)$ for $f \in C^r(\qgrp{G}), a \in A$ and $k \in \qgrp{K}$.

The above construction and the correspondence
\begin{equation}
\label{eq:equiv-alg-ind}
\Hom_{\qgrp{K}-\text{alg}}(A, B) \rightarrow \Hom_{\qgrp{G}-\text{alg}}(\Ind^{\qgrp{G}}_{\qgrp{K}} A, \Ind^{\qgrp{G}}_{\qgrp{K}} B) , \quad f \mapsto (\Id_{C^r(\qgrp{G})} \otimes f)|_{\Ind^{\qgrp{G}}_{\qgrp{K}} A}
\end{equation}
of equivariant homomorphisms define a functor $\Ind^{\qgrp{G}}_{\qgrp{K}}$ from the category of $\qgrp{K}$-algebras to that of $\qgrp{G}$-algebras.

Let $A_1, A_2$ be $\qgrp{G}$-algebras whose coactions are denoted by $\alpha_i\colon A_i \rightarrow C^r(\qgrp{G}) \otimes A_i$ for $i = 1, 2$.  Recall that a cycle in $\KK^{\qgrp{G}}(A_1, A_2)$ is represented by a triple $(\bimod{E}, X, F)$ where $(\bimod{E}, F)$ is a Kasparov $A_1$-$A_2$-bimodule, and $X \in \End(C^r(\qgrp{G}) \otimes \bimod{E})_{C^r(\qgrp{G})}$ is a corepresentation of $C^r(\qgrp{G})$ on $\bimod{E}$, satisfying the equivariance condition
\begin{gather*}
(X (1 \otimes \xi), X (1 \otimes \eta))_{C^r(\qgrp{G}) \otimes A_2} = \alpha_2((\xi, \eta)_{A_2}),\\
\begin{align*}
X (1 \otimes \xi a_2) &= (X (1 \otimes \xi)) \alpha_2(a_2),&
X (1 \otimes a_1 \xi) &= \alpha_1(a_1) X  1 \otimes \xi,
\end{align*}
\end{gather*}
and the almost invariance condition
\[
(1 \otimes a_1) (1 \otimes F - X^*(1 \otimes F) X) \in C^r(\qgrp{G}) \otimes \cptops(\bimod{E}_{A_2}),
\]
where $a_i \in A_i$ and $\xi \in \bimod{E}$.  Let $h_\qgrp{G}$ be the Haar state of $C^r(\qgrp{G})$, and put $\tilde{F} = h_\qgrp{G} \otimes \iota(X^* (1 \otimes F) X)$.  As in the case of ordinary compact groups, the $\qgrp{G}$-invariant operator $\tilde{F}$ defines an $A_1$-$A_2$-Kasparov bimodule which is in the same class as $F$.

Let $(\bimod{E}, X, F)$ be an equivariant Kasparov $B_1$-$B_2$-bimodule satisfying $[X, 1 \otimes F] = 0$.  The induced Kasparov bimodule can be described as follows.  First, put
\[
\Ind^{\qgrp{G}}_{\qgrp{K}} \bimod{E} = \ensemble{\xi \in C^r(\qgrp{G}) \otimes \bimod{E} \suchthat R_{\qgrp{K}} \otimes \iota (\xi) = \iota \otimes X(\xi_{1 3}) }.
\]
Then the restriction of the $C^r(\qgrp{G}) \otimes B_2$-valued inner product to $\Ind^{\qgrp{G}}_{\qgrp{K}} \bimod{E}$ takes value in $\Ind^{\qgrp{G}}_{\qgrp{K}} B_2$ by Proposition~\ref{prop:qgrp-coaction-ind-as-subalg-tens-prod}.  The restriction of $\Delta \otimes \iota$ defines a $\qgrp{G}$-equivariant bimodule structure on this bimodule.  Furthermore, the $\qgrp{K}$-invariance of $F$ implies that $1 \otimes F$ preserves $\Ind^{\qgrp{G}}_{\qgrp{K}} \bimod{E}$ and defines an $\Ind^{\qgrp{G}}_{\qgrp{K}} B_1$-$\Ind^{\qgrp{G}}_{\qgrp{K}} B_2$-Kasparov bimodule.  This procedure defines a functor $\Ind^{\qgrp{G}}_{\qgrp{K}}$ from $\KK^{\qgrp{K}}$ to $\KK^{\qgrp{G}}$ which extends~\eqref{eq:equiv-alg-ind}.  By the uniqueness of such extension, this agrees with the induction functor between the equivariant $\KK$-categories defined in~\cite{MR2566309}.

By the regularity of $\qgrp{G}$, the restriction of $\qgrp{G}$-coactions to $\qgrp{K}$-coactions makes sense.  It induces the the restriction functor $\Res^{\qgrp{G}}_{\qgrp{K}}\colon \KK^\qgrp{G} \rightarrow \KK^{\qgrp{K}}$.  Then $\Ind^{\qgrp{G}}_\qgrp{K}$ is right adjoint to $\Res^{\qgrp{G}}_{\qgrp{K}}$, in the sense that we have a natural isomorphism
\begin{equation}
\label{eq:qgrp-KK-res-ind-adjunction}
\KK^\qgrp{K}(\Res^{\qgrp{G}}_\qgrp{K}(A), B) \simeq \KK^{\qgrp{G}}(A, \Ind^{\qgrp{G}}_\qgrp{K}(B))
\end{equation}
for any $\qgrp{G}$-algebra $A$ and $\qgrp{K}$-algebra $B$~\cite{MR2566309}*{Proposition~4.7}.

When $A$ is a $\qgrp{K}$-Yetter--Drinfeld algebra, $\Ind^{\qgrp{G}}_{\qgrp{K}} A$ has a structure of a $\qgrp{G}$-Yetter--Drinfeld algebra~\cite{MR2566309}*{Theorem~3.4}.  In view of~\eqref{eq:yetter-drinfeld-condi-by-action} and~\eqref{eq:induction-for-cpt-subgrp}, the action of $\C[\qgrp{G}]$ on $D$ given by
\[
(f \otimes a) . \psi = (\psi_{(1)} f S(\psi_{(3)})) \otimes (a * \psi_{(2)}), \quad (f \otimes a \in D, \psi \in \C[\qgrp{G}])
\]
characterizes the coaction of $C^*(\qgrp{G})$ on $\Ind^{\qgrp{G}}_{\qgrp{K}} A$.  Here, the right action of $\C[\qgrp{G}]$  on $A$ is induced by the coaction of $C^*(\qgrp{K})$ and the embedding $C^*(\qgrp{K}) \rightarrow \multAlg(C^*(\qgrp{G}))$.

Applying this to the $\qgrp{K}$-Yetter--Drinfeld algebra $\C$, we obtain that $C^r(\qgrp{G}/\qgrp{K}) = \Ind^{\qgrp{G}}_{\qgrp{K}} \C$ is a $\qgrp{G}$-Yetter--Drinfeld algebra.  We then have a natural isomorphism~\cite{MR2182592}*{Theorem~8.2}
\begin{equation}
\label{eq:res-ind-braided-tensor}
\Ind^{\qgrp{G}}_{\qgrp{K}} \Res^{\qgrp{G}}_{\qgrp{K}}(A) \simeq C^r(\qgrp{G}/\qgrp{K}) \boxtimes_{\qgrp{G}} A
\end{equation}
when $A$ is a $\qgrp{G}$-algebra (this also follows from the description~\eqref{eq:induction-for-cpt-subgrp} of $\Ind^{\qgrp{G}}_{\qgrp{K}}$).

\begin{example}
\label{expl:ind-G_q-is-braided-tensor-flag-G_q}
We may view $C(G_q)$ and $C(T)$ as $T$-algebras by the left translation action of $T$ on these algebras.  For the induction from $C(G_q)$, we have
\begin{equation*}
\Ind^{G_q}_T C(G_q) = C(G_q / T) \boxtimes C(G_q) = (C(G_q) \otimes C(G_q))^{\rho^{-1} \otimes \lambda(T)}
\end{equation*}
by~\eqref{eq:res-ind-braided-tensor} and Remark~\ref{rmk:br-tensor-func-alg-with-itself}.  For the one from  $C(T)$, we have
\[
 \Ind^{G_q}_T C(T) = C(G_q),
\]
by~\eqref{eq:induction-for-cpt-subgrp}.
\end{example}

\section{Equivariant comparison of quantum manifolds}
\label{sec:equivar-compar-quant-g-sp}

%

\subsection{Quantum symplectic leaves}
\label{sec:quant-sympl-leaves}

Let $\alpha$ and $\gamma$ be the standard generators of $C(\SU_q(2))$ satisfying
\begin{align*}
\alpha^* \alpha + \gamma^* \gamma &= 1, &
\alpha \alpha^* + q^2 \gamma \gamma^*  &= 1, &
\gamma^* \gamma &= \gamma \gamma^*, &
\alpha \gamma &= q \gamma \alpha, &
\alpha \gamma^* &= q \gamma^* \alpha.
\end{align*}
There is a C$^*$-algebra homomorphism from $C(\SU_q(2))$ onto $C(\overline{\Disc}_q)$ defined by
\begin{align}
\label{eq:alg-hom-SUq2-to-q-disc}
\alpha & \mapsto Z_q^*, &
\gamma & \mapsto - (1 - Z_q Z_q^*)^{\frac{1}{2}}.
\end{align}
By composing this with~\eqref{eq:q-disc-Toeplitz-corr}, we obtain a representation $\rho_q$ of $C(\SU_q(2))$ on $\ell^2 \N$.

The homomorphism~\eqref{eq:alg-hom-SUq2-to-q-disc} restricts to an isomorphism between $C(\SU_q(2)/\U(1))$ and the unitization $C_0(\Disc_q)^+$ of $C_0(\Disc_q)$.  We record the image of generators of $C(\SU_q(2)/\U(1))$ under this map:
\begin{align*}
\alpha \alpha^* &\mapsto Z_q^* Z_q, &
\alpha \gamma^* &\mapsto - Z_q^* (1 - Z_q Z_q^*)^{\frac{1}{2}}, &
\gamma \gamma^* &\mapsto (1 - Z_q Z_q^*).
\end{align*}

Let $\sigma_t$ be the action of $\U(1)$ on $C(\SU_{q}(2))$ defined by
\begin{align*}
\sigma_t(\alpha) &= e^{2 \pi i t} \alpha, &
\sigma_t(\gamma) &= \gamma.
\end{align*}
It is a rescaled action of the modular automorphism group of the Haar state.  It is also related to the translation actions by $\sigma_{2 t}(x) = L_t R_t (x)$.  Similarly, the action
\begin{align*}
\tau_t(\alpha) &= \alpha, &
\tau_t(\gamma) &= e^{2 \pi i t} \gamma
\end{align*}
of $\U(1)$ corresponds to the scaling group, and satisfies $\tau_{2 t}(x) = L_t R_{-t}(x)$.

Let $U_t$ be the unitary representation of $\U(1)$ on $\ell^2 \N$ defined by
\begin{equation}
\label{eq:torus-action-on-ell-2-N}
U_t e_n = e^{2 \pi i n t} e_n \quad (n \in \N).
\end{equation}
Via~\eqref{eq:q-disc-Toeplitz-corr}, the adjoint action $\Adj_{U_t}$ on $\Tpltz$ is identified with the rotation action $Z_q \mapsto e^{-2 \pi i t} Z_q$ of $\U(1)$ on $C(\overline{\Disc}_q)$.  It follows that the representation $\rho_q$ is equivariant with respect to $\sigma_t$ and $\Ad_{U_t}$.

\subsection{Equivariant comparison of quantum discs}
\label{sec:compar-fam-q-discs}

\begin{lem}
\label{lem:q-disc-equivar-equiv-pt}
Let $k$ be a positive integer.  For each $q \in I$, the embedding
\begin{equation}
\label{eq:q-disc-prod-scal-emb}
\iota\colon \C \rightarrow \underbrace{C(\overline{\Disc}_{q}) \otimes \cdots \otimes C(\overline{\Disc}_{q})}_{k \times}
\end{equation}
induces a $\KK^{\U(1)^k}$-equivalence.
\end{lem}

\begin{proof}
Suppose that the assertion is verified for $k = 1$, with the inverse $\beta$ of $\iota$ in $\KK^{\U(1)}(C(\overline{\Disc}_q), \C)$.  Then the morphism $\beta^{\otimes k}$ will be the inverse of~\eqref{eq:q-disc-prod-scal-emb}.  Hence it is enough to prove the assertion for $k = 1$.

We first deal with the case $q = 1$.  The evaluation at the origin $\ev_0$ defines a $\U(1)$-equivariant extension $C_0(\overline{\Disc} \setminus \ensemble{0}) \rightarrow C(\overline{\Disc}) \rightarrow \C$.  Since the kernel $C_0(\overline{\Disc} \setminus  \ensemble{0})$ is equivariantly contractible, $\ev_0$ induces a $\KK^{\U(1)}$-equivalence which is obviously inverse to $\iota$.

Next we consider the case $q < 1$.  Then we may use the following $\U(1)$-equivariant graded Kasparov $\Tpltz$-$\C$-bimodule $\beta$, already considered by Pimsner in the non-equivariant context~\cite{MR1426840}*{Definition~4.3 and Theorem~4.4}.  His argument applies to our setting without change, but we review it for the reader's convenience.

The graded Hilbert space for $\beta$ is given by $\ell^2 \N \oplus \ell^2 \N_+$, endowed with the action of $\U(1)$ defined by~\eqref{eq:torus-action-on-ell-2-N} on each direct summand.  We consider the natural action $\pi$ of $\Tpltz$ on the first summand $\ell^2 \N$.  On the second one, it is given by the truncation $\chi_{\N_+} \pi \chi_{\N_+}$.  The odd Fredholm operator for $\beta$ is given by the transposition of the two summands, and $\delta_0 \oplus 0 \mapsto 0$.

On the one hand, it can be easily seen that the Kasparov product $[\iota] \otimes_{\Tpltz} \beta$ equals the identity morphism of $\KK^T(\C, \C)$.  On the other hand, the morphism $\beta \otimes_{\C} [\iota] - [\Id_{\Tpltz}]$ is represented by the following Hilbert $\Tpltz$-bimodule.  The graded space is $(\ell^2 \N \otimes \Tpltz) \oplus (\ell^2 \N \otimes \Tpltz)$ endowed with the diagonal action $(U_t \otimes \Adj_{U_t})^{\oplus 2}$ of $\U(1)$.  The right action of $\Tpltz$ is given by the right multiplication action on the second tensor component of each direct summand.  The left action of $\Tpltz$ on the first copy of $\ell^2 \N \otimes_{\C} \Tpltz$ is given by $\pi \otimes \Id_{\Tpltz}$, and on the second copy by
\begin{equation}
\label{eq:left-action-on-beta-iota}
x \mapsto \chi_{\ensemble{0}} \otimes \pi_l(x) + (\chi_{\N_+} \pi(x) \chi_{\N_+}) \otimes \Id_{\Tpltz},
\end{equation}
where $\pi_l$ is the left multiplication action of $\Tpltz$ on itself.  By a $\Tpltz$-compact perturbation, the `$\Tpltz$-Fredholm' operator can be given by the permutation of two direct summands.

Now, we construct a continuous homotopy between~\eqref{eq:left-action-on-beta-iota} and $\pi \otimes \Id_{\Tpltz}$.  For each $t \in [0, 1]$, consider the operator
\begin{align*}
S_t \delta_0 \otimes x &= \cos(\frac{\pi}{2}t) \delta_1 \otimes x + \sin(\frac{\pi}{2} t) \delta_0 \otimes S x, &
S_t \delta_k \otimes x &= \delta_{k+1} \otimes x \quad (x \in \Tpltz, k \in \N_+).
\end{align*}
Then one has a homomorphism $\pi_t\colon \Tpltz \rightarrow \mathcal{L}(\mathcal{E}_{\Tpltz})$ by $\pi_t(S) = S_t$ at each $t$.  They satisfy $\pi_t - \pi_{t'} \in \cptops(\mathcal{E}_{\Tpltz})$, and from the choice of $S_t$ above it follows that the connecting morphisms $\pi_t$ remain $\U(1)$-equivariant.  The morphism $\pi_1$ agrees with~\eqref{eq:left-action-on-beta-iota}, while the one $\pi_0$ equals the degenerate representation $(\pi \otimes \Id_{\Tpltz}) \oplus (\pi \otimes \Id_{\Tpltz})$.  Hence we obtain $\beta \otimes_\C \iota - [\Id_\Tpltz] = 0$ in $\KK^{\U(1)}(\Tpltz, \Tpltz)$.
\end{proof}

Consider the action of $\U(1)$ on $\Gamma_I(C(\overline{\Disc}_q))$ defined by
\begin{align*}
\sigma_t(Q) &= Q, &
\sigma_t(Z) &= e^{2 \pi i t} Z.
\end{align*}
This is a fiberwise action over $I$, given by $\Adj_{U_t}$ at each $q \in [0, 1)$ modulo the representation~\eqref{eq:q-disc-Toeplitz-corr}, and the rotation of the disc around the origin at $q = 1$.

In the following proposition we observe that the evaluation map at any point induces an equivariant $\KK$-equivalence.  Note that the $\KK$-equivalence without the torus action was proved in~\cite{arXiv:1103.4346}*{Lemma~6.3}.

\begin{prop}
\label{prop:U1-equiv-quantum-disc-bundle-eval}
Let $k \in \N$ and $i_1, \ldots i_k$ be positive integers.  For any $q_0 \in I$, the evaluation map $\ev_{q_0}$ for the $\U(1)^k$-$C(I)$-algebra 
\begin{equation}
\label{eq:open-disc-prod-bundle}
\Gamma_{I}\left(C_0(\Disc_{q^{i_1}})  \otimes \cdots \otimes C_0(\Disc_{q^{i_k}})\right)
\end{equation}
is a $\KK^{\U(1)^k}$-equivalence.  The $K_0^{\U(1)^k}$-group of this algebra is the free $R(\U(1)^k)$-module of rank one.
\end{prop}

\begin{proof}
Without the loss of generality we may simplify to the case $i_1 = \cdots = i_k = 1$.  The assertion is equivalent to that the kernel of $\ev_{q_0}$ is $\KK^{\U(1)^k}$-contractible.  When $q_0 = 1$, the kernel of $\ev_{1}$ is the space of $C_0$-sections of the constant field over $[0, 1)$ with fiber $C_0(\Disc_0)^{\otimes k}$, which is equivariantly contractible.

Suppose that $q_0 < 1$.  We have the decomposition
\[
\ker(\ev_{q_0}) = \ker(\ev_{q_0} |_{\Gamma_{[0, q_0]}(C_0(\Disc_q))}) \oplus \ker(\ev_{q_0} |_{\Gamma_{[q_0, 1]}(C_0(\Disc_q))}).
\]
By the triviality of the field $(C_0(\Disc_q))_{q \in [0, 1)}$, the first summand of the right hand side is equivariantly contractible.  By reparametrization, the left hand side is isomorphic to $\ker(\ev_0)$.  It follows that we may suppose $q_0 = 0$.  

For $0 \le m \le k$, let $A_{k,m}$ denote the $\U(1)^k$-$C(I)$-algebra
\begin{equation*}
\Gamma_{I}\big(\underbrace{C_0(\Disc_q) \otimes \cdots \otimes C_0(\Disc_q)}_{m \times}
\otimes
\underbrace{ C(\overline{\Disc}_q) \otimes \cdots \otimes C(\overline{\Disc}_q)}_{ (k - m) \times}\big).
\end{equation*}
Its fiber at $q$ is denoted by $A^{(q)}_{k, m}$.  This algebra is isomorphic to $C_0(\Disc_q)^{\otimes m} \otimes C(\overline{\Disc}_q)^{\otimes (k-m)}$.

We argue by induction on $k$ and $m$ that $K^{\U(1)^k}(A^{(q)}_{k, m}) \simeq R(\U(1)^k)$ and that the evaluation map $\ev_0$ induces an isomorphism of $R(\U(1)^k)$-modules.

First, let us consider the case $m = 0$.  By the first part of this proof and Lemma~\ref{lem:q-disc-equivar-equiv-pt}, the $K^{\U(1)^k}_0$-groups of both $\Gamma_{I}(C(\overline{\Disc}_q)^{\otimes k})$ and $C(\overline{\Disc}_0)^{\otimes k}$ are isomorphic to $R(\U(1)^k)$.  The evaluation map sends the class of unit of $\Gamma_I(C(\overline{\Disc}_q)^{\otimes k})$ to that of $C(\overline{\Disc}_0)^{\otimes k}$, which are basis of the $K^{\U(1)^{\otimes k}}_0$-groups.

Now, the $\U(1)^k$-algebra $A_{k,0}$ is in the equivariant bootstrap class by Lemma~\ref{lem:q-disc-equivar-equiv-pt}.  Since it has an $R(\U(1)^k)$-projective $K^{\U(1)^k}_*$-group, we can apply the equivariant Universal Coefficient Theorem~\cite{MR849938}*{Theorem~10.1} to conclude that $\ev_0$ is a $\KK^{\U(1)^k}$-equivalence when $m = 0$.

Next, suppose that the assertion was proved for the continuous fields $A_{k',m'}$ satisfying $k' < k$ and $A_{k, m'}$ for $m' \le m$.
Let us first verify that the map $\ev_{0}$ induces an isomorphism between the $\U(1)^k$-equivariant $K_0$-groups of $A_{k,m}$ and $A^{(0)}_{k,m}$.  We have the extension
\begin{equation}
\label{eq:A-k-m-induction-diag}
 \begin{CD}
 A_{k,m+1} @>>> A_{k,m} @>>> B \\
 @V{\ev_0}VV @V{\ev_0}VV @V{\ev_0}VV \\
 A^{(0)}_{k,m+1} @>>> A^{(0)}_{k,m} @>>> B^{(0)}
 \end{CD},
\end{equation}
where $B$ is the $\U(1)^k$-$C(I)$-algebra defined by
\begin{equation*}
B = \Gamma_{I}\big (\underbrace{C_0(\Disc_q) \otimes \cdots \otimes C_0(\Disc_q)}_{m \times} \otimes C(S^1) \otimes \underbrace{C(\overline{\Disc}_q) \otimes \cdots \otimes C(\overline{\Disc}_q)}_{ (k - m - 1) \times} \big ).
\end{equation*}

Now, the evaluation map for $B$ is the tensor product of the maps $\ev_0$ for the $C(I)$-algebras $A_{k-1,m}$ and $C(S^1) \otimes C(I)$.  By the induction assumption, the former is a $\KK^{\U(1)^{k-1}}$-isomorphism.  The latter is trivially an $\KK^{\U(1)}$-isomorphism.  Hence $\ev_0$ for $B$ is an $\KK^{\U(1)^k}$-isomorphism.  Next, the evaluation for $A_{k, m}$ is again an $\KK^{\U(1)^k}$-isomorphism by the induction assumption.

Since the algebras $B$ and $B^{(0)}$ are nuclear, the rows in the diagram~\eqref{eq:A-k-m-induction-diag} are split extensions of $\U(1)^k$-algebras.  Hence these can be identified with (parts of) mapping cone triangles~\cite{MR2193334}*{Section~2.3}.  Since the category $\KK^{\U(1)^k}$ is an triangulated category with the mapping cone triangles as its exact triangles~\cite{MR2193334}*{Proposition~2.1}, the evaluation map for $A_{k, m+1}$ is also an $\KK^{\U(1)^k}$-isomorphism.

It remains to show that the $\U(1)^k$-equivariant $K$-group of $A_{k, m+1}$ is a free $R(\U(1)^k)$-module of rank one.  On the one hand, the Green--Julg isomorphism implies
\[
K^{\U(1)^k}_*(B^{(0)}) \simeq K_*(\U(1)^k \ltimes B^{(0)}) \simeq K_*(\U(1)^{k-1} \ltimes A_{k-1,m}) \otimes \U(1) \ltimes C(S^1).
\]
The right hand side of this formula is isomorphic to $K^{\U(1)^{k-1}}_*(A_{k-1, m})$, which is in turn isomorphic to $R(\U(1)^{k-1})$ as an $R(\U(1)^{k-1})$-module by the induction hypothesis.  The structure of the $R(\U(1)^k)$-module on it is induced by the group homomorphism
\[
\phi_m\colon \U(1)^k \rightarrow \U(1)^{k-1}, (t_1, \ldots, t_k) \mapsto (t_1, \ldots, t_m, t_{m+2}, \ldots, t_k).
\]
On the other hand, by the induction hypothesis for $A_{k,m}$, we also know that $K^{\U(1)^k}_*(A_{k,m}^{(0)})$ is isomorphic to $R(\U(1)^k)$.  Applying the functor $K^{\U(1)^k}_*$ to the bottom row of~\eqref{eq:A-k-m-induction-diag}, we obtain a $6$-term exact sequence
\[
\begin{CD}
K^{\U(1)^k}_0(A_{k, m+1}) @>>> R(\U(1)^k) @>{(\phi_m)_*}>> R(\U(1)^{k-1}) \\
@AAA & & @VVV \\
0 @<<< 0 @<<< K^{\U(1)^k}_1(A_{k, m+1})
\end{CD}.
\]
From this we obtain $K^{\U(1)^k}_0(A_{k, m+1}) \simeq R(\U(1)^k)$ and $K^{\U(1)^k}_1(A_{k, m+1}) \simeq 0$, which proves the assertion.
\end{proof}

\subsection{Equivariant comparison of quantum homogeneous spaces}
\label{sec:equivar-fam-of-quant-homog-sp}

For each simple root $\alpha_i$, there is a $*$-Hopf algebra homomorphism $\UnivEnv_{q_i}(\liealg{sl}_2) \rightarrow \UnivEnv_q(\liealg{g})$ defined by
\begin{align*}
E &\mapsto E_i, &
F &\mapsto F_i, &
K &\mapsto K_i.
\end{align*}
Its transpose extends to a $*$-homomorphism $\sigma_i\colon C(G_q) \rightarrow C(\SU_{q_i}(2))$.  Then $\pi_i = \rho_{q_i} \sigma_i$ is a representation of $C(G_q)$ on $\ell^2 \N$.

By construction the operators generated by the torus $\U(1)$ in $\UnivEnv(\SU_{q_i}(2))$ is mapped to the subalgebra generated by $K_i$ in $\UnivEnv(G_q)$.

Let $a_i$ denote the $i$-th column vector $(a_{j, i})_{j \in \Pi}$ of the Cartan matrix of $\liealg{g}$.  It defines a homomorphism $\alpha_i\colon T \rightarrow \U(1)$.  By the relations~\eqref{eq:q-env-def-torus-rel} and $q_i^{a_{i, j}} = q_j^{a_{j, i}}$, the homomorphism $\sigma_i$ intertwines the adjoint action of $T$ on $C(G_q)$ and the one $\tau_{\alpha_i(t)}$ on $C(\SU_{q_i}(2))$ by
\begin{equation}
 \label{eq:sigma-i-equivariance}
 \sigma_i(\Adj_t(f)) = \tau_{\alpha_i (t)}(\sigma_i(f)).
\end{equation}

For each element $w$ of $W$, we shall choose and fix an reduced presentation
\[
w = s_{i_1(w)} \cdots s_{i_k(w)}.
\]
Then we obtain an irreducible representation $\pi_w$ of $C(G_q)$ on $\ell^2 \N^{\otimes k}$ defined by
\[
\pi_w = (\pi_{i_1(w)} \otimes \cdots \otimes \pi_{i_k(w)}) \Delta_q^{(k-1)}.
\]
For the neutral element $e \in W$, the representation $\pi_e$ is understood to be the counit map $\epsilon\colon C(G_q) \rightarrow \C$.  Finally, we put $\pi_{w, t} = \pi_w \circ R_t$ for each $t \in T$.  The representations $(\pi_{w, t})_{w \in W, t \in T}$ give a complete parametrization of the irreducible representations of $C(G_q)$~\cite{MR1614943}.

Let us describe the relationship between $L_t$ and $\pi_{s_i, t}$.  By~\eqref{eq:sigma-i-equivariance} and the equivariance of $\rho_{q_i}$ for $\sigma_t$ and $\Adj_{U_t}$, we obtain
 \[
\pi_i(L_t(f)) = \Adj_{U_{\alpha_i (t)}} \pi_i R_{2 \iota_i (\alpha_i (-t)) + t}(f),
 \]
 where $\iota_i$ is the inclusion of $\U(1)$ in $T$ corresponding to the subgroup $e^{i h_i} \subset T$.  The expression $s_i(t) = 2 \iota_i (\alpha_i (-t)) + t$ is nothing but the standard action of $s_i$ on $T$, which extends to an action of $W$.

Iterating the above argument using $\Delta_q^{(m-1)} L_t = (L_t \otimes \iota^{\otimes m-1}) \Delta_q^{(m-1)}$ and $(R_t \otimes \iota) \Delta_q = (\iota \otimes L_t) \Delta_q$, we obtain the following lemma.

\begin{lem}[c.f.~\cite{arXiv:1103.4346}*{Lemma~3.4}]
\label{lem:quant-schubert-cell-transl}
Let $0 < q < 1$, $w \in W$ and $t \in T$ be given.  For $1 \le k \le m = l(w)$, put $x_k = \alpha_{i_k}(s_{i_{k-1}(w)} \cdots s_{i_1(w)}(t))$.  Then we have
\begin{equation}
\label{eq:sigma-w-equivariance-tau}
\pi_w (L_t(f))
= \left( (\Adj_{U_{x_1}} \circ \pi_{i_1(w)}) \otimes \cdots \otimes (\Adj_{U_{x_m}} \circ \sigma_{i_m(w)})\right) \Delta_q^{(k-1)}(R_{w(t)}f).
\end{equation}
\end{lem}

As in Section~\ref{sec:poisson-subgrp}, let $S$ be a subset of $\Pi$ and $L$ be a subgroup of $P^c(S)$.  Let $W^S$ denote the set of elements $w \in W$ satisfying $w(\alpha_i) \in P$ for all $\alpha_i \in S$.  For each $0 \le m \le m_0$, put
\[
\ideal{J}^{(q)}_m = \cap \ensemble{\ker \pi_{w,t} \cap C(G_q/K^{S,L}) \suchthat w \in W^S, l(w) = m, t \in T}.
\]
It is a $T \times (T/T_L)$-invariant bilateral ideal of $C(G_q/K^{S,L})$.  By~\cite{arXiv:1103.4346}*{Theorem~3.1}, we obtain a composition sequence
\begin{equation}
\label{eq:q-flag-compos-seq}
0 = \ideal{J}^{(q)}_{m_0} \subset \ideal{J}^{(q)}_{m_0 - 1} \subset \cdots \subset \ideal{J}^{(q)}_0 \subset \ideal{J}^{(q)}_{-1} = C(G_q/K^{S,L}_q),
\end{equation}
satisfying
\begin{equation}
\label{eq:q-flag-compos-seq-subquot}
\ideal{J}^{(q)}_{m-1}/\ideal{J}^{(q)}_m \simeq \bigoplus_{\substack{w \in W^S,\\l(w) = m}} C(T/T_L) \otimes C_0(\Disc_{q_{i_1(w)}}) \otimes \cdots \otimes C_0(\Disc_{q_{i_m(w)}}),
\end{equation}
given by the surjective map $a \mapsto (t \mapsto \pi_{w,t}(a))_w$ from $\ideal{J}^{(q)}_{m-1}$ to the right hand side.

\begin{rmk}
\label{rmk:T-action-adj-and-right-transl}
By the definitions of $\pi_{w,t}$ and the isomorphism~\eqref{eq:q-flag-compos-seq-subquot}, the left translation action of $t \in T$ on $\ideal{J}^{(q)}_{m-1}/\ideal{J}^{(q)}_m$ corresponds to
\begin{equation}
\label{eq:left-transl-on-q-schubert-cells}
\bigoplus_{\substack{w \in W^S,\\l(w) = m}} L^{(T/T_L)}_{w(t)} \otimes \sigma_{\alpha_{i_1(w)}(t)} \otimes \sigma_{\alpha_{i_2(w)}(s_{i_1(w)}t)} \otimes \cdots \otimes \sigma_{\alpha_{i_m(w)}(s_{i_{m-1}(w)} \cdots s_{i_1(w)}t)},
\end{equation}
where $L^{(T/T_L)}_t$ denotes the translation on $T/T_L$ by $t \in T$.  The right translation action of $T/T_L$ induces the translation action on $C(T/T_L)$ and the trivial action on the rest.
\end{rmk}

For the case $q = 1$, we have the embeddings $\gamma_{i}\colon \SU(2) \rightarrow G$ associated to each $i \in \Pi$.  By composing with the embedding of $\Disc$ into $\SU(2)$ given by
\[
z \rightarrow \left (
  \begin{array}{cc}
    \overline{z} & \sqrt{ 1 - \absolute{z}^2} \\
    - \sqrt{ 1 - \absolute{z}^2} & z
  \end{array} \right ),
\]
we obtain embeddings of discs
\[
\gamma_w \colon \Disc^k \rightarrow G, \quad (z_1, \ldots, z_k) \mapsto \gamma_{i_1(w)}(z_1) \cdots \gamma_{i_k(w)}(z_k).
\]
 The adjoint action of $T$ on $\Img(\gamma_w)$ is given by the following action determined by the maps $\alpha_{i_1(w)}, \ldots , \alpha_{i_k(w)}$ which is linear with respect to the coordinate given by $\gamma_w$~\cite{MR1726697}*{Proposition~5}:
\begin{multline}
\label{eq:gamma-w-equivariance}
t. \gamma_w(z_1, \ldots, z_k) = \\
\gamma_w\left (\alpha_{i_1(w)}(t) . z_1, \alpha_{i_2(w)}(s_{i_{1}(w)} t) . z_k \ldots, \alpha_{i_k(w)}(s_{i_{k-1}(w)} \cdots s_{i_{1}(w)} t) . z_k \right ).w(t).
\end{multline}
Thus, we may interpret Lemma~\ref{lem:quant-schubert-cell-transl} as an analogue of this formula in the quantum setting.

For each $w \in W$, we consider a continuous $T$-$C(I)$-algebra
\[
A_w = \Gamma_{I}\big (C_0(\Disc_{q_{i_1(w)}}) \otimes \cdots \otimes C_0(\Disc_{q_{i_k(w)}})\big)
\]
whose fiber is given by
\[
A^{(q)}_w = C_0(\Disc_{q_{i_1(w)}}) \otimes \cdots \otimes C_0(\Disc_{q_{i_k(w)}}) = \Img( \pi_w),
\]
endowed with the action of $T$ given by
\[
\sigma_{\alpha_{i_1(w)}(t)} \otimes \sigma_{\alpha_{i_2(w)}(s_{i_1(w)}t)} \otimes \cdots \otimes \sigma_{\alpha_{i_m(w)}(s_{i_{m-1}(w)} \cdots s_{i_1(w)}t)}
\]
for $q < 1$ and by~\eqref{eq:gamma-w-equivariance} at $q = 1$.

\begin{prop}
\label{prop:q-Disc-prod-bundle-eval-equivar-KK-equiv}
 Let $w \in W$ and $J$ be a closed subinterval of $I$.  Then the evaluation map $\Gamma_J(A^{(q)}_w) \rightarrow A^{(q_0)}_w$ induces a $\KK^T$-equivalence for each $q_0 \in J$.
\end{prop}

\begin{proof}
Since the field $(A^{(q)}_w)_{q \in I}$ is trivial over $[0, 1)$, we may restrict to the case $J = I$.  The $T$-$C(I)$-algebra $A_w$ is given by~\eqref{eq:open-disc-prod-bundle} with respect to the integer sequence
\[
\left (\frac{\left ( \alpha_{i_1(w)}, \alpha_{i_1(w)} \right )}{2}, \ldots, \frac{\left (\alpha_{i_k(w)}, \alpha_{i_k(w)} \right )}{2} \right ).
\]
Moreover, the action of $\T$ is induced by the homomorphism
\[
T \rightarrow \U(1)^k, \quad t \mapsto (\alpha_{i_1(w)}(t), \alpha_{i_2(w)}(s_{i_1(w)}t), \ldots , \alpha_{i_m(w)}(s_{i_{m-1}(w)} \cdots s_{i_1(w)}t)).
\]
Hence we obtain the assertion by Proposition~\ref{prop:U1-equiv-quantum-disc-bundle-eval}.
\end{proof}

We have seen in Section~\ref{sec:cont-field-quant-homog-sps} that the C$^*$-algebras $C(G/K^{S, L}_q)$ form a continuous field in an essentially unique way.  This can be considered as a field of $T \times (T/T_L)$-algebras coming from the fiberwise left and right translations of $T$.

Regarding this field, we obtain the following $T\times (T/T_L)$-equivariant equivalence for the evaluation maps.  The equivalence without the action of $T \times (T/T_L)$ was already proved in~\cite{arXiv:1103.4346}*{Theorem~6.1}.

\begin{thm}
\label{thm:q-grp-bundle-eval-T-T-equivar-equiv}
 Let $J$ be a closed subinterval of $(0, 1]$ and $q_0 \in J$.  Then the evaluation map
 \[
 \ev_{q_0}\colon \Gamma_{J}(C(G_q/K^{S,L}_q)) \rightarrow C(G_{q_0}/K^{S,L}_{q_0})
 \]
 induces a $\KK^{T \times (T/T_L)}$-equivalence.
\end{thm}

\begin{proof}
 We first note that the ideals $\ideal{J}^{(q)}_m$ for $q \in J$ form a $T \times (T/T_L)$-invariant continuous subfield of $(C(G_q/K^{S, L}_q))_{q \in J}$.  We argue by induction that the evaluation map $\Gamma_{J}(\ideal{J}^{(q)}_m) \rightarrow \ideal{J}^{(q_0)}_m$ is an equivalence of the $\KK^{T \times (T/T_L)}$-category.
 
 The case for $m = m_0$ trivially holds as $\ideal{J}^{(q)}_{m_0}$ is the zero ideal.  Now, suppose that $\ev_{q_0}$ is a $\KK^T$-equivalence for the bundle $(\ideal{J}^{(q)}_m)_{q \in J}$.  From~\eqref{eq:q-flag-compos-seq-subquot}, one obtains the equivariant exact sequence
 \begin{equation}
 \label{eq:compos-seq-bundle}
0 \rightarrow \Gamma_{J}(\ideal{J}^{(q)}_m) \rightarrow \Gamma_{J}(\ideal{J}^{(q)}_{m-1}) \rightarrow \bigoplus_{\substack{w \in W^S,\\l(w) = m}} C(T/T_L) \otimes \Gamma_{J}(A^{(q)}_w) \rightarrow 0.
 \end{equation}
This extension is compatible with the evaluation homomorphism $\ev_{q_0}$ of each field.
 
 By Proposition~\ref{prop:q-Disc-prod-bundle-eval-equivar-KK-equiv} and Remark~\ref{rmk:T-action-adj-and-right-transl}, $\ev_{q_0}$ is a $\KK^{T \times (T/T_L)}$-equivalence for the quotient of~\eqref{eq:compos-seq-bundle}.  Since the quotient is nuclear, we have a completely positive splitting of this extension, and it can be identified with a (part of) mapping cone triangle.  The induction hypothesis for the middle column, together with the fact that $\KK^{T \times (T/T_L)}$ is a triangulated category imply that the evaluation map of $\Gamma_J(\ideal{J}^{(q)}_{m-1})$ is also a $\KK^{T \times (T/T_L)}$-equivalence.
\end{proof}

\begin{rmk}
\label{rmk:right-action-RKK-fib}
 S. Neshveyev pointed out the following partial strengthening of the above result.  Let $H$ be a locally compact group.  We say that a continuous field of $H$-algebras $A = (A_q)_{q \in J}$ on an interval $J$ is an $\RKK^H$-fibration~\cite{MR2511635} if $A$ is $\RKK^H$-equivalent to the constant field with fiber $A_q$ for any $q$.  Then the field $(C(G_q/K^{S,L}_q))_q$ is an $\RKK^{T/T_L}$-fibration with respect to the right translation action.  Indeed, Proposition~\ref{prop:q-Disc-prod-bundle-eval-equivar-KK-equiv} and~\cite{MR2511635}*{Corollary~1.6} implies that $\Gamma_{J}(A^{(q)}_w)$ is an $\RKK$-fibration.  Hence the quotient field in~\eqref{eq:compos-seq-bundle} is an $\RKK^{T/T_L}$-fibration.  Since the equivariant $\RKK$-category is triangulated and splitting exists as a $T/T_L$-$C(I)$-linear completely positive map~\cite{MR1610242}*{Th\'{e}or\`{e}me~7.2}, we may use the induction on the ideals $\Gamma_{J}(\ideal{J}^{(q)}_m)$ as above.
\end{rmk}

\section{Applications}
\label{sec:applications}

\subsection{Ring structure of \texorpdfstring{$K^*(C(G_q))$}{K(C(Gq))}}
\label{sec:ring-stru-K-C-Gq}

The coproduct homomorphism $\Delta_q \colon C(G_q) \rightarrow C(G_q) \otimes C(G_q)$ induces a map from $K^*(C(G_q) \otimes C(G_q))$ to $K^*(C(G_q)$.  Combining this with the external Kasparov product
\[
\KK(C(G_q), \C) \times \KK(C(G_q), \C) \rightarrow \KK(C(G_q) \otimes C(G_q), \C),
\]
we obtain an associative product
\begin{equation}
\label{eq:coprod-prod-on-homol}
x \cdot y = \Delta_q \otimes_{C(G_q) \otimes C(G_q)} (y \otimes x)
\end{equation}
on the $K$-homology group $K^*(C(G_q))$.  

The continuous field $(C(G_q))_{q \in (0, 1]}$ can be thought as a field of Hopf C$^*$-algebras in the following way.  Let $(\phi^q)_{q \in (0, 1]}$ be a continuous family of isomorphisms from $\UnivEnv(G_q)$ to $\UnivEnv(G)$, and $J$ be a closed subinterval of $(0, 1]$.  Now, we can define a $*$-algebra homomorphism from $C(J) \otimes \C[G]$ into $\Gamma_{J}(C(G_q) \otimes C(G_q))$ by $f \otimes a \mapsto (f(q) \Delta_q(\phi^{q\#}(a)))_{q \in J}$.  The norm on $C(J) \otimes \C[G]$ induced by this homomorphism also satisfies the condition of the $C(J)$-algebra norm on $C(J) \otimes \C[G]$ in Section~\ref{sec:cont-field-quant-homog-sps}.  Hence it induces a $C(J)$-algebra homomorphism
\[
\Delta_*\colon \Gamma_J( C(G_q)) \rightarrow \Gamma_J(C(G_q) \otimes C(G_q)).
\]

\begin{thm}
\label{thm:ring-structure-G-q}
The ring $K^*(C(G_q))$ is mutually isomorphic for $q \in (0, 1]$.
\end{thm}

\begin{proof}
Let $J$ be an arbitrary closed subinterval of $(0, 1]$.  For each $t \in J$, we let $\Ev_t$ denote the functor from $\RKK(J; -, -)$ to $\KK$ given by the evaluation at $t$.  By Remark~\ref{rmk:right-action-RKK-fib},
\begin{equation}
\label{eq:eval-RKK-homol-to-KK-homol}
\Ev_t\colon \RKK(J; \Gamma_J(C(G_q)), C(J)) \rightarrow \KK(C(G_t), \C)
\end{equation}
is an isomorphism of abelian groups for any $t \in J$.

For $x$ and $y$ in $\RKK(J; \Gamma_J(C(G_q)), C(J))$, put
\[
x \cdot y = \Delta_* \otimes_{\Gamma_J(C(G_q) \otimes C(G_q))} (y \otimes_{C(J)} x).
\]
This defines a product structure on $\RKK(J; \Gamma_J(C(G_q)), C(J))$ such that~\eqref{eq:eval-RKK-homol-to-KK-homol} becomes a ring homomorphism.
\end{proof}

\begin{rmk}
The ring $K^*(C(G))$ is the $K$-homology group of the compact topological space $G$ endowed with the ring structure induced by the group law map $G \times G \rightarrow G$.
\end{rmk}

\subsection{Induction-restriction and the \texorpdfstring{$K$}{K}-homology ring struture}

The above ring structure can be described in terms of the induction procedure and operations on the $T$-equivariant $\KK$-groups.  The starting point is the Green--Julg isomorphism~\cite{vergnioux-thesis}*{Proposition~5.11}
\begin{equation}
\label{eq:green-julg-CG_q-homol-T-G_q-T}
K^*(C(G_q)) \simeq \KK^T_*(\Res^{G_q}_T C(G_q), C(T)).
\end{equation}

Let $\eta\colon C(G_q) \rightarrow \Ind^{G_q}_T C(G_q)$ be the structure morphism of adjunction~\eqref{eq:qgrp-KK-res-ind-adjunction} for the case of $A = C(G_q)$.  This homomorphism can be identified with the embedding of $C(G_q)$ into the second component of $C(G_q/T) \boxtimes_{G_q} C(G_q)$~\cite{MR2566309}*{the proof of Proposition~4.7}.  Unwrapping the identification~\eqref{eq:green-julg-CG_q-homol-T-G_q-T}, we see that the product structure on $K^*(C(G_q))$ corresponds to the operation
\[
x \cdot y =  [\eta] \otimes_{\Ind^{G_q}_T C(G_q)} \Res^{G_q}_T \Ind^{G_q}_T(x) \otimes_{C(G_q)} y
\]
on $\KK^T(C(G_q), C(T))$.
 
Carrying out the above consideration within the framework of continuous field, we can give an alternative, although seemingly unnecessarily complicated, proof of Theorem~\ref{thm:ring-structure-G-q}.  We sketch the argument in the rest of the section.

Let us fix a closed subinterval $J$ of $(0, 1]$.  Then, we can consider the notion of a $\Gamma_J(C(G_q))$-algebra, which is given by a $C(J)$-algebra $A$ endowed with the `coaction' map
\[
\alpha\colon A \rightarrow \Gamma_J(C(G_q)) \otimes_{C(J)} A
\]
satisfying $\alpha_{1 3} \alpha_{2 3} = (\Delta_*)_{1 2} \alpha_{2 3}$.  For example, the algebras
\[
\Gamma_J(C(G_q)), \quad \Gamma_J(C(G_q) \boxtimes_{G_q} C(G_q)), \quad \Gamma_J(C(G_q/T) \boxtimes_{G_q} C(G_q))
\]
have natural structures of $\Gamma_J(C(G_q))$-algebra.  The $G_q$-algebra homomorphisms $(\eta \colon C(G_q) \rightarrow C(G_q/T) \boxtimes C(G_q))_{q \in J}$ patch together and define a $\Gamma_J(C(G_q))$-algebra homomorphism
\[
\eta_*\colon \Gamma_J(C(G_q)) \rightarrow \Gamma_J(C(G_q/T) \boxtimes C(G_q)).
\]

The continuous field $(C(G_q))_{q \in J}$ admits the `fiberwise right translation' coaction
\[
R_{T\mhyphen C(J)}\colon \Gamma_J(C(G_q)) \rightarrow \Gamma_J(C(G_q) \otimes C(T))
\]
by $T$.  When $A$ is a $T$-$C(J)$-algebra given by the coaction $\alpha\colon A \rightarrow C(T) \otimes A$, we can define its induction to a $\Gamma_J(C(G_q))$-algebra defined by
\[
\Ind^{\Gamma_J(C(G_q))}_{T\mhyphen C(J)} A = \ensemble{x \in \Gamma_J(C(G_q)) \otimes_{C(J)} A \suchthat (R_{T\mhyphen C(J)})_{1 2}(x) = \alpha_{2 3}(x)}.
\]
Furthermore, the endofunctor $A \mapsto \Res_T \Ind^{\Gamma_J(C(G_q))}_{T\mhyphen C(J)} A$ on the category of $T$-$C(J)$-algebras can be extended to an endofunctor on $\RKK^T(J; -, -)$.\footnote{It should be able to define the `fiberwise $G_q$-equivariant $\KK$-category' $\KK^{\Gamma_J(C(G_q))}$ which serves as the receptacle of the factorization $\Ind^{\Gamma_J(C(G_q))}_{T\mhyphen C(J)}\colon \RKK^T(J; -, -) \rightarrow \KK^{\Gamma_J(C(G_q))}$ and $\Res^{\Gamma_J(C(G_q))}_{T\mhyphen C(J)}\colon \KK^{\Gamma_J(C(G_q))} \rightarrow \RKK^T(J; -, -)$, but we do not dare to do that in this paper.}

An argument similar to Remark~\ref{rmk:right-action-RKK-fib} shows that $\Gamma_J(C(G_q))$ is an $\RKK^T$-fibration with respect to the left translation by $T$.  Thus,
\[
\Ev_t\colon \RKK^T(J; \Gamma_J(C(G_q)), C(T) \otimes C(J)) \rightarrow \KK^T(C(G_p), C(T))
\]
is an isomorphism of abelian groups for any $t \in J$.

When $x$ and $y$ are elements in $\RKK^T(J; \Gamma_J(C(G_q)), C(T) \otimes C(J))$, put
\[
x \cdot y = \eta_* \otimes_{\Gamma_J(C(G_q/T) \boxtimes C(G_q))} \Res_T \Ind^{\Gamma_J(C(G_q))}_{T\mhyphen C(J)} (x) \otimes_{\Gamma_J(C(G_q))} y.
\]
It defines a ring structure on $\RKK^T(J; \Gamma_J(C(G_q)), C(T) \otimes C(J))$ such that $\Ev_t$ becomes a ring homomorphism for each $t$. This proves Theorem~\ref{thm:ring-structure-G-q}.

\subsection{The Borsuk--Ulam theorem for the quantum spheres}
\label{sec-borsuk-ulam-q-sphere}

Let $C_2$ denote the cyclic group of order $2$.  We consider the $n$-sphere $S^n$ as a $C_2$-space with respect to the antipodal map $x \mapsto -x$, where we consider $S^n$ as the unit sphere of $\R^{n+1}$.  The Borsuk--Ulam theorem states that there is no $C_2$-equivariant continuous map from $S^{n+1}$ to $S^n$.  Since Borsuk's initial proof, there have been numerous alternative proofs and generalizations of the problem.  See~\citelist{\cite{MR801938}\cite{MR1988723}} for an overview.

In this section we show that the odd-dimensional quantum spheres $S^{2n+1}_q$ of Example~\ref{example:quantum-sphere-as-homogen-su-n-su-n-1} and the even-dimensional ones $S^{2n}_q$ introduced by Hong--Szyma\'{n}ski~\cite{MR1942860} satisfy an analogous statement, as conjectured by Baum--Hajac~\cite{baum-hajac-galois-klein}.

Let us briefly review the construction of quantum spheres.  The algebra $C(S^{2n-1}_q)$ is generated by the elements $z_1, \ldots, z_n$ satisfying
\begin{gather*}
z_j z_i = q z_i z_j \quad (i < j), \quad z_j^* z_i = q z_i z_j^* \quad (i \neq j),\\
z_i^* z_i^{} = z_i z_i^* + (1-q^2)\sum_{j>i} z_j z_j^*, \quad \sum_{i=1}^n z_i z_i^* = 1.
\end{gather*}
The algebra $C(S^{2n}_q)$ of the $2n$-dimensional quantum sphere is defined as the quotient of $C(S^{2n+1}_q)$ by the ideal generated by $z_{n+1} - z_{n+1}^*$~\cite{MR1942860}*{Section~5}.

There is also a surjective homomorphism from $C(S^{2n}_q)$ to $C(S^{2n-1}_q)$.  These surjections $C(S^{n+1}_q) \rightarrow C(S^n_q)$, denoted by $\iota^\#$, correspond to the embedding of $S^n$ as an equator in $S^{n+1}$.

The map $z_i \rightarrow -z_i$ for $i = 1, \ldots, n+1$ defines an action of $C_2$ on $C(S^{2n+1}_q)$, which descends to $C(S^{2n}_q)$.  The homomorphisms $\iota^\#$ are $C_2$-equivariant.

The $K$-groups of the quantum spheres are computed as~\citelist{\cite{MR1086447}\cite{MR1942860}}
\[
K_*(C(S^{2n-1}_q)) \simeq \Z \oplus \Z, \quad K_*(C(S^{2n}_q)) \simeq \Z^2 \oplus 0.
\]
Thus, these invariants are isomorphic to the classical case.

The action of $T \times (T/T_L) \simeq U(1)^n$ on $C(S^{2n-1}_q)$ is given by
\[
(\lambda_1, \ldots, \lambda_n).(z_1, \ldots, z_n) = (\lambda_n \lambda_1 z_1, \lambda_n \overline{\lambda_1} \lambda_2 z_2, \lambda_n \overline{\lambda_2} \lambda_3 z_3, \ldots \lambda_n \overline{\lambda_{n-1}} z_n).
\]
Hence it is the separate gauge action up to a finite cover.

\begin{prop}
\label{prop:equivar-q-sphere-equiv-classic-sphere}
 Let $n$ be any integer greater than $1$.  The algebra $C(S^{2n-1}_q)$ of odd quantum sphere is $\KK^{U(1)^n}$-equivalent to $C(S^{2n-1})$.
\end{prop}

\begin{proof}
 This follows from Theorem~\ref{thm:q-grp-bundle-eval-T-T-equivar-equiv}, applied to the Poisson--Lie subgroup $\SU(n-1)$ of $\SU(n)$ (see Example~\ref{example:quantum-sphere-as-homogen-su-n-su-n-1}).
\end{proof}

Let $A$ be a C$^*$-algebra whose $K$-groups have finite rank over $\Z$.  The \textit{Lefschetz number} $\Lef(x)$~\cite{MR2797970} of an element $x$ in $\KK_0(A, A)$ is the alternating trace of the linear map on the graded vector space $K_*(A) \otimes \Q$ induced by $x$.

\begin{prop}
  \label{prop:free-action-lefschetz-num}
Let $H$ be a finite group, and $M$ be a closed manifold endowed with a free action of $H$.  Then, for any element $x$ of $\KK^H(C(M), C(M))$, the Lefschetz number of $x$ is an integer divisible by the order of $H$.
\end{prop}

\begin{proof}
  We may assume that $M$ is a Riemannian manifold endowed with an $H$-invariant metric.  Let $\Cliff(M)$ be the associated Clifford algebra bundle over $M$, and let $\Gamma\Cliff(M)$ denote the $\Z/2\Z$-graded C$^*$-algebra of its continuous sections.  Then by functoriality $\Gamma\Cliff(M)$ becomes an $H$-algebra.

There are elements
\[
\hat{\Delta} \in \KK^H(\C, \Gamma\Cliff(M) \grtensor C(M)), \quad \Delta \in \KK^H(C(M) \grtensor \Gamma\Cliff(M), \C)
\]
which satisfy the Poincar\'{e} duality equations~\cite{MR918241}
\begin{align*}
(\hat{\Delta} \otimes \Id_{\Gamma\Cliff(M)}) \otimes_{\Gamma\Cliff(M) \grtensor C(M) \grtensor \Gamma\Cliff(M)} (\Id_{\Gamma\Cliff(M)} \otimes \Delta) &= \Id_{\Gamma\Cliff(M)} \\
\intertext{and}
(\Id_{C(M)} \otimes \hat{\Delta}) \otimes_{C(M) \grtensor \Gamma\Cliff(M) \grtensor \Gamma\Cliff(M)} (\Delta \otimes \Id_{C(M)}) &= \Id_{C(M)}.
\end{align*}
We recall that $\Delta$ is defined as the composition of the product map
\[
m \colon C(M) \grtensor \Gamma\Cliff(M) \rightarrow \Gamma\Cliff(M)
\]
and the element $[D] \in \KK^H(\Gamma\Cliff(M), \C)$ represented by an $H$-equivariant elliptic operator $D$ on the graded Clifford module $\wedge^\even(M) \oplus \wedge^\odd(M)$.

Let $\Sigma$ be the flip map from $C(M) \grtensor \Gamma\Cliff(M)$ to $\Gamma\Cliff(M) \grtensor C(M)$.  By~\cite{MR2797970}*{Theorem~6}, we have the equality
\begin{equation}
  \label{eq:lef-num-compute-poincare}
  \Lef(x) \Id_\C = \hat{\Delta} \otimes_{\Gamma\Cliff(M) \grtensor C(M)} (\Id_{\Gamma\Cliff(M)} \otimes x) \otimes_{\Gamma\Cliff(M) \grtensor C(M)} \Sigma \Delta
\end{equation}
for any $x \in \KK(C(M), C(M))$.

Now, let $\xi$ be a Clifford module representing the element
\[
\hat{\Delta} \otimes_{\Gamma\Cliff(M) \grtensor C(M)} (\Id_{\Gamma\Cliff(M)} \otimes x) \otimes_{\Gamma\Cliff(M) \grtensor C(M)} \Sigma m  \in \KK^H(\C, \Gamma\Cliff(M)).
\]
Then the right hand side of~\eqref{eq:lef-num-compute-poincare} is equal to the index of the twist $D_\xi$ of $D$ by $\xi$, which is again an $H$-equivariant elliptic operator on $M$.  Thus, we obtain an elliptic operator  $\widetilde{D_\xi}$ on the quotient manifold $M/H$ induced by $D_\xi$.  We have the equality $\Ind(D_\xi) = \absolute{H} \Ind(\widetilde{D_\xi})$ as in~\cite{MR0420729}, which proves the assertion.
\end{proof}

\begin{thm}
\label{thm:quant-borsuk-ulam}
Let $n$ be any positive integer.  For any value of $0 < q \le 1$, there is no $\KK^{C_2}$-morphism from $C(S^n_q)$ to $C(S^{n+1}_q)$ which induces a unital homomorphism on the $K_0$-groups.
\end{thm}

\begin{proof}
 The antipodal map on $C(S^{2k-1}_q)$ can be realized as the action of $-1 \in \U(1)$ with respect to the right translation action on $C(\SU_q(k)/\SU_q(k-1))$.  By Proposition~\ref{prop:equivar-q-sphere-equiv-classic-sphere}, we have a $\KK^{C_2}$-equivalence between $C(S^{2k-1}_q)$ and $C(S^{2k-1})$.  By Proposition~\ref{prop:free-action-lefschetz-num}, the Lefschetz number of any $x \in \KK^{C_2}(C(S^{2k-1}_q), C(S^{2k-1}_q))$ must be an even integer.
 
 Now, suppose that $n$ is odd, and that $\phi$ is a $C_2$-equivariant $\KK^{C_2}$-morphism from $C(S^n_q)$ to $C(S^{n+1}_q)$ which is unital on $K_0$.  Then the composition $\phi \otimes_{C(S^{n+1}_q)} [\iota^\#]$ is a $C_2$-equivariant morphism from $C(S^n_q)$ to itself.  The map on $K_*(C(S^n_q))$ induced by $\phi \otimes_{C(S^{n+1}_q)} [\iota^\#]$ factors through $K_*(C(S^{n+1}_q))$, which is trivial in the odd part.  Hence the graded trace of $\phi \otimes_{C(S^{n+1}_q)} [\iota^\#]$ on $K_*(C(S^n_q))$ becomes equal to $1$, which is a contradiction.

Next, suppose that $n$ is even, and let $\phi$ be a $C_2$-equivariant $\KK^{C_2}$-morphism from $C(S^n_q)$ to $C(S^{n+1}_q)$ which is unital on $K_0$.  Then the composition $[\iota^\#] \otimes_{C(S^{n}_q)} \phi$ is a $C_2$-equivariant morphism on $C(S^{n+1}_q)$.  Analogously to the above argument, we obtain a contradiction in this case as well.
\end{proof}

The following is an analogue of the Borsuk--Ulam theorem for quantum spheres.

\begin{cor}
 For any $0 < q \le 1$ and any integers $n < m$, there is no $C_2$-equivariant unital $*$-homomorphism from $C(S^n_q)$ to $C(S^m_q)$.
\end{cor}

\begin{proof}
We argue by contradiction.  Suppose that there is a $C_2$-equivariant unital $*$-homomorphism $\psi$ from $C(S^n_q)$ to $C(S^m_q)$.  Then, by composing with iterations of $\iota^\#$, we obtain a equivariant unital homomorphism $\phi$ from $C(S^n_q)$ to $C(S^{n+1}_q)$.  The class of $\phi$ in $\KK^{C_2}(C(S^n_q), C(S^{n+1}_q))$ induces a unital homomorphism on the $K_0$-groups, which contradicts the above theorem.
\end{proof}

\begin{rmk}
It also follows from Theorem~\ref{thm:quant-borsuk-ulam} that different dimensional quantum spheres cannot be $C_2$-equivariantly equivalent.  (This also follows from the explicit computation of $K_*(C(S^n_q)^{C_2})$ carried out in~\cite{MR1942860}.)  Indeed, when $\phi$ is a $\KK^{C_2}$-equivalence between $C(S^m_q)$ and $C(S^n_q)$, the parity of $m$ and $n$ has to be the same.  If $m$ is odd, clearly either $\phi$ or $-\phi$ has to be unital on $K_0(C(S^m_q)) = \Z [1]$.  When $m$ is even, using the extension~\cite{MR1942860}*{(5.1)} for $\iota^\#$, the subgroup $\Z[1]$ of $K_0(C(S^m_q))$ can be characterized as the trivial submodule of the induced action of $C_2$ on $K_0(C(S^m_q))$.  Thus $\phi$ or $-\phi$ has to be unital on $K_0$ in this case as well.
\end{rmk}

\begin{rmk}
A classical result of Gottlieb~\cite{MR687962} states that, when $M$ is a free $G$-manifold which is a homotopy retract of a finite CW-complex and $f$ is an equivariant self-continuous map on $M$, the Lefschetz number of $f$ is divisible by the order of $G$.  Proposition~\ref{prop:free-action-lefschetz-num} can be seen as a generalization of this to the equivariant $\KK$-morphisms in the compact manifold setting.  It would be interesting to know to which extent Proposition~\ref{prop:free-action-lefschetz-num} can be generalized.  For example, a result of Casson and Gottlieb~\cite{MR0436144}*{Theorem~4} establishes an analogous statement for equivariant self continuous maps of compact spaces of finite covering dimension endowed with a free action of a finite cyclic group.  This implies that there is no $C_2$-equivariant continuous map from the suspension $S X$ to $X$ when $X$ is such a free $C_2$-space.
\end{rmk}





\end{document}